\documentclass[pdflatex]{amsart}

\usepackage[usenames]{color}
\usepackage{amssymb}
\usepackage{graphicx, epsfig, hyperref}
\usepackage{latexsym, amsfonts, amscd, amsmath}
\usepackage{mathrsfs}
\makeindex 
\input xy
\xyoption{all}

\definecolor{Indigo}{rgb}{0.2,0.1,0.7}
\definecolor{Violet}{rgb}{0.5,0.1,0.7}

\newtheorem{thm}{Theorem}[section]
\newtheorem{prop}[thm]{Proposition}
\newtheorem{lem}[thm]{Lemma}
\newtheorem{cor}[thm]{Corollary}

\newtheorem{dfn}[thm]{Definition}

\theoremstyle{remark}
\newtheorem{rmk}[thm]{Remark}

\newenvironment{mat}{\left(\begin{smallmatrix}
\end{smallmatrix}\right)}{enddef}

\numberwithin{equation}{section} \numberwithin{figure}{section}
\numberwithin{table}{section}

\newcommand{\Aut}{{\operatorname{Aut}}}

\newcommand{\Cl}{{\operatorname{Cl}}}
\newcommand{\Cent}{{\operatorname{Cent}}}

\newcommand{\diag}{{\operatorname{diag}}}

\newcommand{\disc}{{\operatorname{disc}}}
\newcommand{\Emb}{{\operatorname{Emb}}}
\newcommand{\End}{{\operatorname{End}}}
\newcommand{\Fr}{{\operatorname{Fr }}}
\newcommand{\Hom}{{\operatorname{Hom}}}

\newcommand{\Ima}{{\operatorname{Im}}}

\newcommand{\Ker}{{\operatorname{Ker}}}

\newcommand{\Norm}{{\operatorname{Norm }}}

\newcommand{\ord}{{\operatorname{ord }}}

\newcommand{\Span}{{\operatorname{Span }}}

\newcommand{\Tr}{{\operatorname{Tr }}}
\newcommand{\Trd}{{\operatorname{Trd}}}
\newcommand{\val}{{\operatorname{val}}}

\newcommand{\SL}{{\operatorname{SL }}}
\newcommand{\M}{{\operatorname{M }}}

\newcommand{\Gal}{{\operatorname{Gal}}}

\newcommand{\Qpbar}{{\overline{\mathbb{Q}}_p}}

\newcommand{\gera}{{\mathfrak{a}}}
\newcommand{\gerb}{{\mathfrak{b}}}
\newcommand{\gerc}{{\mathfrak{c}}}
\newcommand{\gerd}{{\mathfrak{d}}}

\newcommand{\gerf}{{\mathfrak{f}}}

\newcommand{\gerl}{{\mathfrak{l}}}
\newcommand{\germ}{{\mathfrak{m}}}

\newcommand{\gerp}{{\mathfrak{p}}}
\newcommand{\gerq}{{\frak{q}}}
\newcommand{\gerr}{{\frak{r}}}

\newcommand{\gerH}{{\frak{H}}}

\newcommand{\gerP}{{\frak{P}}}

\newcommand{\uA}{{\underline{A}}}

\newcommand{\alphabar}{{\overline{\alpha}}}
\newcommand{\betabar}{{\overline{\beta}}}
\newcommand{\gammabar}{{\overline{\gamma}}}
\newcommand{\deltabar}{{\overline{\delta}}}
\newcommand{\pibar}{{\overline{\pi}}}
\newcommand{\mubar}{{\overline{\mu}}}

\newcommand{\calA}{{\mathcal{A}}}

\newcommand{\calD}{{\mathcal{D}}}

\newcommand{\calH}{{\mathcal{H}}}

\newcommand{\calO}{{\mathcal{O}}}

\def\CC{\mathbb{C}}
\def\DD{\mathbb{D}}

\def\FF{\mathbb{F}}

\def\HH{\mathbb{H}}

\def\QQ{\mathbb{Q}}
\def\RR{\mathbb{R}}

\def\WW{\mathbb{W}}

\def\ZZ{\mathbb{Z}}

\newcommand{\scrF}{{\mathscr{F}}}

\newcommand{\id}{{\noindent}}

\newcommand{\arr}{{\; \rightarrow \;}}

\newcommand{\injects}{{\; \hookrightarrow \;}}

\newcommand{\ol}{{\mathcal{O}_L}}
\newcommand{\oleta}{{\mathcal{O}_{L_\eta}}}
\newcommand{\oketa}{{\mathcal{O}_{K_\eta}}}

\newcommand{\ok}{{\mathcal{O}_K}}

\newcommand{\fpbar}{{\overline{\FF}_{p}}}

\begin{document}
\marginparwidth 50pt

\title{A Gross-Zagier formula for quaternion algebras over totally real fields}

\author{Eyal Z. Goren \& Kristin E. Lauter}
\address{Department of Mathematics and Statistics,
McGill University, 805 Sherbrooke St. W., Montreal H3A 2K6, QC,
Canada.}\address{Microsoft Research, One Microsoft Way, Redmond,
WA 98052, USA.} \email{goren@math.mcgill.ca;
klauter@microsoft.com} 

\begin{abstract}
We prove a higher dimensional generalization of Gross and Zagier's theorem on the factorization of differences of singular moduli.
Their result is proved by giving a counting formula for the number of isomorphisms between elliptic curves with complex multiplication by two different imaginary quadratic fields $K$ and $K^\prime$, when the curves are reduced modulo a supersingular prime and its powers.  Equivalently, the Gross-Zagier formula counts optimal embeddings of the ring of integers of an imaginary quadratic field into particular maximal orders in $B_{p, \infty}$, the definite quaternion algebra over $\QQ$ ramified only at $p$ and infinity. Our work gives an analogous counting formula for the number of simultaneous embeddings of the rings of integers of primitive CM fields into superspecial orders in definite quaternion algebras over totally real fields of strict class number $1$.  Our results can also be viewed as a counting formula for the number of isomorphisms modulo $\frak{p} \vert p$ between abelian varieties with CM by different fields.
Our counting formula can also be used to determine which superspecial primes appear in the factorizations of differences of values of Siegel modular functions at CM points associated to two different CM fields, and to give a bound on those supersingular primes which can appear.  In the special case of Jacobians of genus $2$ curves, this provides information about the factorizations of numerators of Igusa invariants, and so is also relevant to the problem of constructing genus $2$ curves for use in cryptography.
\end{abstract}

\maketitle

\section{Introduction}

The celebrated theorem of Gross and Zagier~\cite{Gross Zagier} gives a factorization of norms of differences of singular moduli: values of the modular $j$-function evaluated at CM points associated to imaginary quadratic fields.  Let $K$ and $K^\prime$ be two imaginary quadratic fields with relatively prime fundamental discriminants $d$ and $d^\prime$.
For $\tau$ and $\tau^\prime$ running through equivalence classes of imaginary quadratic integers in the upper half plane modulo $\SL_2(\ZZ)$ with  $\disc(\tau) =d$, $\disc(\tau^\prime) = d^\prime$, and $w$ and $w^\prime$ equal to the number of roots of unity in $K$ and $K^\prime$ respectively, define 
$$J(d,d^\prime)= (\prod_{[\tau], [\tau^\prime]} (j(\tau)-j(\tau^\prime)))^{\frac{4}{w w^\prime}}.$$
Then the theorem of Gross and Zagier states that if $\lambda$ is a prime of $\ok$ of characteristic $p$, then
$$ord_{\lambda}J(d,d^\prime) = \frac{1}{2} \sum_{x\in \ZZ} \sum_{n \ge 1} \delta(x) R \left( \frac{d d^\prime-x^2}{4p^n} \right),$$
where $R(m)$ is the number of ideals of $\ok$ of norm $m$, and $\delta(x)=1$ unless $x$ is divisible by $d$, in which case it is $2$.
Their results can also be viewed as a counting formula for the number of isomorphisms between the reductions modulo primes and their powers of elliptic curves with complex multiplication by two different imaginary quadratic fields $K$ and $K^\prime$.  This in turn is equivalent to counting optimal embeddings of the ring of integers of an imaginary quadratic field into particular maximal orders in $B_{p, \infty}$, the definite quaternion algebra over $\QQ$ ramified only at $p$ and infinity.  Gross and Zagier 
gave an algebraic proof of this result under the additional assumption that $d$ is prime, 
and the algebraic proof of the theorem was extended to arbitrary fundamental, relatively prime discriminants
in a series of papers by Dorman~\cite{DormanOrders},\cite{Dorman1},\cite{Dorman2}.

In this paper we prove a generalization to higher dimensions of Gross and Zagier's theorem, which can also be viewed in three ways as 1) a statement about primes in the factorization of differences of values of Siegel modular functions at CM points associated to two different CM fields;
2) a counting formula for isomorphisms modulo $p$ between abelian varieties with CM by different fields; and 3) a counting formula for simultaneous embeddings of the rings of integers of two primitive CM fields into superspecial orders in certain definite quaternion algebras over a totally real field.

First we explain our interest in these three contexts.  Assume throughout that $K$ and $K^\prime$ are primitive CM fields with a common totally real subfield $K^+={K^\prime}^+=L$ and $[L:\QQ]=g$ where $L$ has strict class number $1$.
In the special case of $g=2$ we are inspired by some concrete calculations of values of certain Siegel modular functions at CM points associated to primitive quartic CM fields.
Let $C$ and $C^\prime$ be two genus $2$ curves whose Jacobians $J$ and $J^\prime$ have complex multiplication
(CM) by $K$ and $K^\prime$.
In analogy with the modular $j$-invariant for elliptic curves, for genus $2$ curves Igusa defined
$10$ modular invariants.  Equality of these $10$ invariants determines whether two curves are isomorphic geometrically, so
primes appearing in the factorization of all $10$ differences correspond to primes where the curves become isomorphic when reduced modulo that prime. Concrete calculations and the tables of van Wamelen suggest that such primes are ``small''.  An explicit characterization of such primes gives information about the numerators of Igusa invariants and thus has some value computationally as well.

Thus we are led to be interested in counting the number of isomorphisms modulo various primes and their powers between abelian varieties with CM by two different CM fields $K$ and $K^\prime$.
The existence of an isomorphism modulo $p$ between abelian varieties with CM by two different CM fields $K$ and $K^\prime$ with $K^+={K^\prime}^+$ implies supersingular reduction modulo $p$.
Fixing an abelian variety A with CM by $K$, each isomorphism modulo $p$ with an abelian variety $A'$ with CM by $K^\prime$ gives an embedding of $\calO_{K^\prime}$ into $\End_\ol(A)$.
In the case of superspecial reduction, we can give very explicit descriptions of the orders $\End_\ol(A)$, which allows us to derive a formula which counts such embeddings.

Goren and Nicole have introduced the notion of superspecial orders in definite quaternion algebras over totally real fields as a generalization of maximal orders in definite quaternion algebras over $\QQ$ (see the thesis of Nicole~\cite{Nicole} and the related paper~\cite{NicoleJNT}).  These orders were further studied in~\cite{CGL1, CGL2, GorenLauterDistance} where related Ramanujan graphs were constructed and certain cryptographic applications suggested.
Throughout this paper assume that $p$ is a prime number which is unramified in the totally real field $L$ of degree $g$ and strict class number $h^+(L)=1$.  Under those assumptions a {\it superspecial order} in $B_{p,L}:=B_{p, \infty} \otimes_\QQ L$ is an Eichler order of level $p$.  The connection with geometry is given in the thesis of Nicole, where it is shown that $\End_{\ol}(A)$ is a superspecial order for $A$ a principally polarized superspecial abelian variety with RM over $\overline{\FF}_p$.  Conversely, every superspecial order arises in this way from such an abelian variety $A$.


Next we give an overview of the results of the paper.
The core of the paper is the generalization of Dorman's work constructing and classifying superspecial
orders in $B_{p,L}$ with an optimal embedding of a CM number field $K$ with $K^+=L$.
First, Section~\ref{sec:BpL} is devoted to giving a description of the quaternion algebra $B_{p,L}$ with a fixed embedding of the CM field $K$ for {\it superspecial primes}, i.e. unramified primes $p$ such that an abelian variety with CM by $K$ has superspecial reduction 
modulo a prime $\frak{P} | p$ in a field of definition of the abelian variety.
Sections~\ref{orders} and~\ref{superspecialorders} establish a classification of superspecial orders with an optimal embedding of $K$, giving both an explicit construction of all such superspecial orders and a bijection (up to conjugation by elements of $K^\times$) with the class group of $K$ (Theorem~\ref{classification}).
These three sections together establish the generalization to $g > 1$ of Dorman's work on orders (\cite{DormanOrders}), and fix several gaps in his proofs.

Section~\ref{maincounting} gives a method for counting embeddings by counting elements of 
the superspecial orders with a prescribed trace and norm in a way that generalizes the Gross-Zagier formula.
Our method is very similar to Gross-Zagier's and Dorman's; their results are the special case $g=1$.
To make the link between the algebraic and the geometric sides of the story, we include the
determination of endomorphism rings of superspecial abelian varieties in Section~\ref{sec:endring}.
Section~\ref{sec:geometry} connects the counting formula for isomorphisms between CM abelian varieties with the counting formula for embeddings into superspecial orders.

The main result of the paper is an explicitly computable counting formula for the number of isomorphisms modulo $\frak{P} | p$ between abelian varieties with CM by two different CM fields $K$ and $K^\prime$
with $K^+={K^\prime}^+$ (Theorems~\ref{SS1S2} and~\ref{thm:geometry}).  
This formula can be viewed as an intersection number  
under the assumption that a reasonable lemma in intersection theory holds (see Section~\ref{sec: moduli}).  
Less precisely, we refer to this
value as a ``coincidence number''.  It also has an algebraic interpretation as the number of optimal triples of embeddings of $\ok$ and $\calO_{K^\prime}$ into superspecial orders (see Section~\ref{sec: triples}).

For primes of supersingular reduction for CM abelian varieties, a separate computation of the
endomorphism rings is given in Section~\ref{supersingularorders}.
In Section~\ref{sec:bound}, a volume argument such as was used in~\cite{GL1} is given to establish a bound
on primes $p$ of either supersingular or superspecial reduction where isomorphisms exist modulo $p$ between CM points associated to $K$ and $K^\prime$.
In Section~\ref{sec:example}, an example of two Galois CM fields is  given and all primes dividing the differences of the Igusa invariants are examined and compared with our counting formula.

\section{Preliminaries} \label{prelim}

\subsection{Quadratic Reciprocity for number fields}

Let $L$ be a number field, and $\gamma$ and $\delta$ prime elements of $L$ which are non-associates, such that
$(\gamma\delta,2)=1$.
Define
\begin{equation*}
	\left(\frac{\gamma}{\delta}\right) = \left\{ \begin{array}{cc}
	1 & \mbox{if $\gamma = \square \mod{\delta}$} \\
	-1 & \mbox{else}.
	\end{array} \right.
	\end{equation*}
Let $B = \left(\frac{\gamma, \, \delta}{L}\right)$ be the quaternion algebra over $L$ defined by the elements $\gamma$ and $\delta$.
For any place $\eta$ of $L$, including the infinite places, define
\begin{equation*}
	(\gamma, \delta)_{\eta} = \left\{ \begin{array}{cc}
	1 & \mbox{if $B \otimes_L L_{\eta}$ is split} \\
	-1 & \mbox{else}.
	\end{array} \right.
	\end{equation*}
Then we have the following analogue of Quadratic Reciprocity for the number field $L$.
\begin{prop} \label{QR}

(1) If $\eta$ is a finite prime such that $\eta \nmid 2$, then $(\gamma, \delta)_{\eta} = 1$ if and only if
$x^2 - \gamma y^2 -\delta z^2 = 0 $ has a non-trivial solution modulo ${\eta}$.

(2) If $\eta$ is complex, then $(\gamma, \delta)_{\eta} = 1$.

(3) If $\eta$ is real ($\eta : L \rightarrow \RR$), then $(\gamma, \delta)_{\eta} = 1$ if and only if
$\eta(\gamma) >0$ or $\eta(\delta) >0$. I.e. $(\gamma, \delta)_{\eta} = -1$ if and only if both $\eta(\gamma)$
and $\eta(\delta)$ are negative.

(4) $$\left(\frac{\gamma}{\delta}\right)\left(\frac{\delta}{\gamma}\right)= (-1)^{r(\gamma, \delta)} \cdot \prod_{\eta \mid 2} (\gamma, \delta)_{\eta},$$
where $r(\gamma, \delta)$ equals the number of real places $\eta$ such that both $\eta(\gamma)$ and $\eta(\delta)$ are negative.
In particular, if either $\gamma$ or $\delta$ are totally positive, then
$$ \left(\frac{\gamma}{\delta}\right)\left(\frac{\delta}{\gamma}\right)=  (\gamma, \delta)_{2} := \prod_{\eta \mid 2} (\gamma, \delta)_{\eta}.$$

(5) We have $$\left(\frac{-1}{\gamma}\right)(-1,\gamma)_2 = (-1)^{r(\gamma)},$$
where $r(\gamma)$ is the number of real places $\eta$ such that $\eta(\gamma)$ is negative.

\end{prop}
\begin{proof}

(1) By \cite[Chap. II, Cor 1.2]{Vigneras},
$(\gamma, \delta)_{\eta} = 1$ if and only if
$x^2 - \gamma y^2 -\delta z^2 = 0$ has a non-trivial solution in $L_{\eta}$,
where by ``non-trivial'' we mean a solution where at least one of the variables with non-zero coefficients is non-zero.  Suppose that $x^2 - \gamma y^2 -\delta z^2 = 0$ has a non-trivial solution in $L_{\eta}$.  By multiplying by a common denominator we can assume $x, y, z \in
\mathcal{O}_{L_{\eta}}$ and one of them is a unit.  Then reducing modulo $\eta$ we get a
non-trivial solution to $x^2 - \gamma y^2 -\delta z^2 \equiv 0 \mod{\eta}$.
Conversely, suppose $x^2 - \gamma y^2 -\delta z^2 \equiv 0 \mod{\eta}$ has a non-trivial solution.
By Hensel's lemma, we can lift the solution to $\mathcal{O}_{L_{\eta}}$.

Part (2) is clear and Part (3) follows from {\it loc. cit.} because
$x^2 - \eta(\gamma) y^2 -\eta(\delta) z^2 = 0$ has a non-trivial solution in $\RR^3$
if and only if either $\eta(\gamma) >0$ or $\eta(\delta) >0$.

To prove (4), first note that $(\gamma, \delta)_{\gamma} = 1 \iff$
$x^2 - \gamma y^2 -\delta z^2 = 0$ has a non-trivial solution modulo ${\gamma}$
$\iff$ $\delta =(\frac{x}{z})^2$ for some non-zero $x, z \in \ol/(\gamma)$
$\iff$ $(\frac{\delta}{\gamma})=1$.
By the product formula:
$$1=\prod_{\eta} (\gamma, \delta)_{\eta} = (-1)^{r(\gamma, \delta)}(\gamma, \delta)_{2}
\left(\frac{\delta}{\gamma}\right) \left(\frac{\gamma}{\delta}\right)
\prod_{\eta { \text{ finite, } } \eta \nmid 2\gamma\delta}(\gamma, \delta)_{\eta}.$$
But for $\eta \nmid 2\gamma\delta$,
$x^2 - \gamma y^2 -\delta z^2 = 0$ has a non-trivial solution modulo ${\eta}$, so
$(\gamma, \delta)_{\eta}=1$.

Similarly for (5), for any real place $\eta$, $\eta(\gamma) >0 \iff (-1,\gamma)_{\eta} = 1$,
so it follows from the product formula that
$$1=\prod_{\eta} (-1, \gamma)_{\eta} = (-1)^{r(\gamma)}\left(\frac{-1}{\gamma}\right)(-1, \gamma)_{2} .$$
\end{proof}

\subsection{The ring of integers in CM fields}

Let $K$ be a CM field with totally real subfield $K^+=L$. Assume that $L$ has strict class number one.  Let $\calD_{K/L}$ be the different of the extension and let $\eta$ denote a prime ideal of $\ol$.
\begin{lem} \label{Lemma: CMdisc} 

(1) $\ok = \ol[t]$, where $t^2+at+b=0$ for some $a,b \in \ol$, and $\calD_{K/L}^{-1} = (\frac{1}{\sqrt{d}})$, with $d=a^2-4b$ a totally negative element of $\ol$.

(2) Assume for $\eta \mid 2$ that if $\eta \mid a$ then $b$ is not a square modulo $\eta$.
Then $(d,2)=1$ and $d$ is square-free.

\end{lem}

\begin{proof} Part (1) is proved in~\cite[Lemma 3.1]{GL2}.

Part (2).  Since $\ok = \ol[t]/(t^2+at+b)$, the prime decomposition of every prime $\eta$ is determined by the prime factorization of $t^2+at+b \mod{\eta}$.  If $\eta$ is ramified, that implies that
$t^2+at+b =(t-c)^2 \mod{\eta}$ for some $c \in \ol/(\eta)$.  But since $\eta \mid 2$, we have
$$(t-c)^2 = t^2 -c^2 = t^2+c^2 \mod{\eta},$$ so
$$t^2+at+b = (t-c)^2 \mod{\eta} \iff \eta \mid a  \, \, {\rm and } \, \, b =\square \mod{\eta}.$$
Thus our condition implies that $\ok$ is unramified over all primes $\eta \mid 2$.
It follows that $(d,2)=1$.

Next we prove that $d$ is square-free.  Let $\eta$ be a prime of $\ol$ not dividing $2$.  For $\eta \mid d$,
we have $\ok \otimes_\ol {\mathcal{O}_{L_\eta}} = \oleta[\sqrt{d}]$ because $\ok =\ol[\frac{-a+\sqrt{d}}{2}]$.
Write $\oleta[\sqrt{d}] = \oleta[\sqrt{u \cdot \alpha_\eta^r}]$, where $u$ is a unit at $\eta$ and
$\alpha_\eta^r \mid d$.  If $r>1$, then
$$\oleta[\sqrt{u \cdot \alpha_\eta^r}] = \oleta + \oleta \cdot \sqrt{u \cdot \alpha_\eta^r}$$ has
no element of valuation $1$, which is not possible.  Indeed, if $\pi$ is a uniformizer
of $\oketa$, with valuation normalized so that
$\val_\eta(\oleta) = \ZZ_{\ge 0}$, then $\val_\pi(x) = 2\val_\eta(x) \in 2\ZZ_{\ge 0}$ for $x \in \oleta$, and
$$\val_\pi(\sqrt{u \cdot \alpha_\eta^r}) = \frac{1}{2}\val_\pi({u \cdot \alpha_\eta^r}) =
\val_\eta({u \cdot \alpha_\eta^r}) = r.$$
In other words, we have shown that discriminants of quadratic extensions of $p$-adic fields are square-free when $p \ne 2$.
\end{proof}

\begin{lem} We have $\ok = \ol[\frac{a'+\sqrt{d}}{2}]$ exactly for the $a' \in \ol$ such that
$a' \equiv a \mod{2\ol}$. Such $a'$ satisfy $(a')^2 \equiv d \mod 4\ol$.  Conversely, given $a' \in \ol$
such that $(a')^2 \equiv d \mod 4\ol$,  we have $\ok = \ol[\frac{a'+\sqrt{d}}{2}]$.
\end{lem}
\begin{proof}
We have $\ok = \ol[t] = \ol[\frac{a+\sqrt{d}}{2}] = \ol[\frac{a'+\sqrt{d}}{2}]$ if
$a' \equiv a \mod{2\ol}$.  We have $d=a^2-4b \equiv a^2 \mod{4\ol}$.
Then also $(a')^2 = (a +2y)^2 = a^2+4ay+4y^2 \equiv d \mod{4\ol}$.

If $\ol[\frac{a+\sqrt{d}}{2}]= \ol[\frac{a'+\sqrt{d}}{2}]$ then
$$\frac{a+\sqrt{d}}{2} = u + v(\frac{a'+\sqrt{d}}{2}),$$ which implies that
$$  a+\sqrt{d} = 2u +va'+v\sqrt{d},$$ and so
$$  v = 1 \quad {\rm and} \quad a=2u+a' \Rightarrow a \equiv a' \mod{2\ol}.$$

Finally, suppose $a' \in \ol$ satisfies $(a')^2 \equiv d \mod {4\ol}$.  Then $\frac{a'+\sqrt{d}}{2}$
is integral.  Therefore
$$\frac{a'+\sqrt{d}}{2} = u + v \cdot(\frac{a+\sqrt{d}}{2})
 \Rightarrow a'+\sqrt{d} = 2u +va+v\sqrt{d} \Rightarrow v = 1 \Rightarrow a \equiv a' \mod{2\ol}.$$
\end{proof}

\subsection{CM points on Hilbert modular varieties} \label{sec:CMpoints} Assume that $L$ is a totally real field, $[L:\QQ] = g$ and $L$ has strict class number $1$; we write $h_L^+ = 1$. This implies that $(\ol^\times)^+ = (\ol^\times)^2$. In this case, the Hilbert modular variety $\calH_L$ associated to $L$ is geometrically irreducible and affords the following description. It is the moduli space for triples $(A, \iota: \ol \arr \End(A), \eta)$, where $A$ is a complex abelian variety of dimension $g$, $\iota$ is a ring embedding and $\eta$ is a principal $\ol$-polarization, or, equivalently, $\eta$ is a principal polarization and the associated Rosati involution fixes $\ol$ element-wise. We have $\calH_L \cong \SL_2(\ol) \backslash \gerH^g$ (see~\cite[Chapter 2, \S 2]{Goren}). Our interest is in the parameterization of CM points on $\calH_L$.

\subsubsection{Abelian varieties with CM} Let $K$ be a CM field such that $K^+ = L$. We consider triples
\begin{equation}\label{equation: triple} (A, \iota: \ok \arr \End(A), \eta) ,
\end{equation}
such that $A$ is a $g$-dimensional complex abelian variety, $\iota$ is a ring homomorphism and $\eta$ is a principal
$\ok$-polarization, where by that we mean a principal polarization whose associated Rosati involution induces complex conjugation on $K$.

Such datum produces a point on $\calH_L$, namely, the point parameterizing $(A, \iota\vert_\ol, \eta)$. This will be examined later. First we want to classify triples $(A, \iota, \eta)$ as in (\ref{equation: triple}) up to isomorphism.

To a triple $(A, \iota, \eta)$ we may associate a CM type $\Phi$ that records the induced action of $K$ on $T_{A, 0}$, the tangent space to $A$ at the origin. The theory of complex multiplication then asserts the existence of a fractional ideal $\gera$ of $K$ such that
\[ (A, \iota) \cong (\CC^g/\Phi(\gera), \iota_{can}), \]
where $\Phi(\gera)$ is the lattice $\{\varphi_1(a), \dots, \varphi_g(a): a\in \gera\}$ and where $\Phi = \{ \varphi_1, \dots, \varphi_g\}$; $\iota_{can}$ is the canonical action of $\ok$ on that abelian variety, obtained by extending the natural action on $\Phi(\gera)$. Furthermore, the principal polarization $\eta$ is induced from a paring on $K$ of the form
\[ (x, y) \mapsto \Tr_{K/\QQ}(ax\bar y), \]
for some $a\in K$. The conditions on $a$ ensuring the associated polarization, say $\eta_a$, is principal are:
\begin{enumerate}
\item $(a) = (\calD_K\gera\bar\gera)^{-1}$.
\item $\bar a = -a$.
\item $\Ima(\varphi_i(a))>0$, for $i = 1, \dots, g$.
\end{enumerate}
It follows easily that for every $\lambda\in K^\times$ the principally polarized abelian variety associated to $(\Phi, \gera, a)$, in the manner above, is isomorphic to that associated to $(\Phi, \lambda \gera, (\lambda \bar \lambda)^{-1}a)$. Furthermore, any isomorphism of principally polarized abelian varieties $(A, \iota, \eta) \cong (A^\prime, \iota^\prime, \eta^\prime)$ as in (\ref{equation: triple}) arises that way.

Now, given a fractional ideal $\gera$ of $K$, the ideal $\gera\bar\gera$ is of the form $\gerb\ok$ for some fractional ideal $\gerb$ of $L$ and, since $h_L = 1$, we can write $(\gera\bar\gera)^{-1} = \lambda \ok$ for a suitable $\lambda\in L$. The fractional ideal $\calD_K^{-1}$ is of the form $d^{-1/2}\ok$, where $d$ is a totally negative element of $L$. Thus,
\[  (\calD_K\gera\bar\gera)^{-1} = (\lambda d^{-1/2}), \]
and $\overline{\lambda d^{-1/2}} = - \lambda d^{-1/2}$. We are free to change $\lambda$ by any unit $\epsilon \in \ol^\times$. Since $(\ol^\times)^+ = (\ol^\times)^2$, it follows easily that for any choice of signs $s_1, \dots, s_g$ in $\{ \pm 1\}$ there is a unit $\epsilon \in \ol^\times$ such that the sign of $\varphi_i(\epsilon)$ is $s_i$. Since $\Ima(\varphi_i(\epsilon \lambda \sqrt{d}^{-1})) = \varphi_i(\epsilon) \Ima(\varphi_i(\lambda \sqrt{d}^{-1}))$, by choosing $\epsilon$ properly we may arrange that $\Ima(\varphi_i(\epsilon \lambda \sqrt{d}^{-1})) >0$ for all $i = 1, \dots, g$. We have thus shown that for every fractional ideal $\gera$ of $K$, there is a suitable $a$ such that $(\Phi, \gera, a)$ gives a principally polarized abelian variety with CM by $K$.

Let $\gera_1, \dots, \gera_{h_K}$ be representatives for the class group of $K$, $\Cl(K)$. Our discussion so far shows that the isomorphism classes of principally polarized abelian varieties with CM by $\ok$ are in bijection with equivalence classes of the following set
\[ \{(\Phi, \gera_i, a) : \Phi \text{ is a CM type}, 1 \leq i \leq h_k, a \text{ satisfies conditions (i) - (iii) above relative to $\Phi, \gera$} \}.\] The equivalence relation is that
\[ (\Phi, \gera_i, a) \sim (\Phi, \gera_i, \epsilon \bar \epsilon a), \qquad \epsilon \in \ok^\times.\]
Given $(\Phi, \gera_i, a)$ and $(\Phi, \gera_i, b)$ there is a unit $\epsilon_1\in \ok^\times$ such that $b = \epsilon_1a$, because both $a$ and $b$ generate the ideal $(\calD_K\gera_i\bar\gera_i)^{-1}$. Since $\bar a = -a$ and $\bar b = -b$, it follows that $\epsilon_1\in \ol^\times$, and since $\Ima(\varphi(a)) > 0$ and $\Ima(\varphi(b)) > 0$ it follows that $\epsilon_1 \in \ol^{\times, +}$. Using that $ \ol^{\times, +} =  \ol^{\times, 2}$, we conclude that there is an $\epsilon \in \ol$ such that $\epsilon_1 = \epsilon^2 =\epsilon \bar \epsilon$. That is, $(\Phi, \gera_i, a) \sim (\Phi, \gera_i, b)$. We therefore conclude that,
in the strict class number 1 case, isomorphism classes of principally polarized abelian varieties with CM by $K$ and a fixed CM type are parameterized by the ideal classes of $K$.

\subsubsection{CM points on $\calH_L$}
Let $(A, \iota:\ok \arr \End(A))$ be a complex abelian variety with CM by $K$ (so $[K:\QQ] = 2\dim(A)$). Since $h_L^+ = 1$, it carries a unique principal polarization up to isomorphism. Consider $\End_\ol(A)$. We use \cite[Lemma 6, p. 464]{Chai}. In the notation of that Lemma, since $A$ has CM only cases III(a) and IV can arise. But, since we are working over the complex numbers, in fact only case IV can arise, and according to which $A \sim B^n$, where $B$ is of dimension $g/n$ and has CM by a CM field $K_0$ whose totally real subfield $L_0$ is contained in $L$ and satisfies $[L:L_0] = n$. One has $\End^0_L(A) = L \otimes_{L_0}K_0$, which is a CM field according to that Lemma. It follows, because $K$ is primitive, that
$\End^0_L(A) = K$. As a consequence, once a RM structure is specified on $A$, there are precisely two CM structures extending it; if $\iota: \ok \arr \End(A)$ is one of them, the other is $\bar \iota: = \iota \circ \tau$, where $\tau$ is complex conjugation on $K$. If $\iota$ has CM type $\Phi$ then $\bar \iota$ has CM type $\bar \Phi$.  Let $\scrF$ be the set of CM types for $K$.

\begin{prop} \label{prop: cm points}
Define an equivalence relation $\sim$ on $\scrF \times \Cl(K)$ by $(\Phi, [\gera]) \sim (\bar\Phi, [\bar \gera]) ( = (\bar\Phi, [\gera^{-1}])).$ Then the set $\scrF \times \Cl(K)/\sim$ has $2^{g-1} \times \#\Cl(K)$ elements and is in a natural bijection with the $K$-CM points on $\calH_L$, that is, with the points $(A, \iota: \ol \arr \End(A), \eta)$ for which we can extend $\iota$ to an embedding $\ok \arr \End(A)$ whose image is fixed (as a set) by the Rosati involution associated to $\eta$.
\end{prop}

\section{Quaternion algebras over totally real fields} \label{sec:BpL}

Let $L$ be a totally real number field of degree $g$ and strict class number 1.
Let $p$ be a prime number unramified in $L$ and
let $$B_{p,L} := B_{p, \infty} \otimes_{\QQ} L,$$ where~$B_{p, \infty}$ is the rational quaternion algebra ramified at~$p$ and~$\infty$ alone.
Let $$S = \{\gerp \triangleleft \ol \mid \gerp|p \}$$ be the set of prime ideals of $L$ above $p$, and
let $$S_0 = \{\gerp \in S \mid f(\gerp/p) \equiv 1 \mod{2} \}$$
be those with odd residue degree.  The algebra $B_{p,L}$ is ramified precisely
at all infinite places and at the primes $\gerp \in S_0$.

The rest of this section and Sections~\ref{orders} and~\ref{superspecialorders} are devoted to giving a description of the
quaternion algebra $B_{p,L}$, and a classification of some particular orders, under the assumption that  all primes $\gerp \in S \setminus S_0$ split
in $K$ and all primes $\gerp \in S_0$ are inert in $K$.  
First we prove that this assumption is satisfied when $p$
is an unramified prime of superspecial reduction for an abelian variety with CM by $K$.  

\subsection{Splitting behavior in the case of superspecial reduction}

\begin{prop} \label{superspecialsplitting}
Let $p$ be a rational prime, unramified in $K$. Let $A$ be an abelian variety with CM by $\ok$, defined over a number field $M$, with good reduction at a prime ideal $\gerp_M$ of $M$ dividing the rational prime $p$. Assume that $A$ has supersingular reduction modulo $\gerp_M$. Then, every prime in $S_0$ is inert in $K$. Assume further that $A$ has superspecial reduction, then every prime in $S\setminus S_0$ is split in $K$.
\end{prop}
\begin{proof} Since $A$ has supersingular reduction, say $\overline{A}$, $\End_L^0(\overline{A}) \cong B_{p, L} = B_{p, \infty} \otimes_\QQ L$ (\cite[Lemma 6]{Chai}), and so
\[ K \injects B_{p, L}.\]
Therefore, at every prime $\gerP$ of $K$ above a prime $\gerp$ of $L$, the field $K_{\gerP}$ splits the quaternion algebra $B_{p, L} \otimes_L L_\gerp$. The quaternion algebra $B_{p, L}$ is ramified precisely at the primes in $S_0$ and at infinity. Thus, if $\gerp \in S_0$, we find that each $K_\gerP$ is a quadratic field extension of $L_\gerp$, that is, since $p$ is unramified in $K$, all the primes in $S_0$ are inert in $K$.

Assume now that there is a prime $\gerp\in S \setminus S_0$ that is inert in $K$ and let $\gerP$ be the prime of $K$ above $\gerp$. Let us denote the embedding of $\ol$ into $W(\fpbar)$ associated to $\gerp$ by $\{ \varphi_1, \dots, \varphi_f\}$, $f = f(\gerp/p)$, where we may order the embeddings so that $\sigma\circ \varphi_i = \varphi_{i+1}$ and $\sigma$ denotes the Frobenius automorphism. Each embedding $\varphi_i$ is the restriction of two embeddings of $\ok$ into $W(\fpbar)$ that we denote $\psi_i^1, \psi_i^2$, where one is the composition of the other with complex conjugation. Since $\gerP$ is inert over $\gerp$, $\sigma$ still acts transitively on the set $\{ \psi_i^j: i = 1, \dots, f, j = 1, 2\}$.

The Dieudonn\'e module of $\overline{A}$ decomposes as $ D = \oplus_{\gerp\vert p} D(\gerp)$ relative to the $\ol$ structure. Let $H = D(\gerp)$. Then $H$ decomposes further as
\[ H = \oplus_{i=1}^f H(\varphi_i) =  \oplus_{i=1}^f \left( H(\psi_i^1)\oplus H(\psi_i^2)\right), \]
where $H(\varphi_i)$ is a free $W(\fpbar)$-module of rank $2$ on which $\ol$ acts via $\varphi_i$ and it is decomposes into a direct sum of two free $W(\fpbar)$-module of rank $1$, $H(\psi_i^1), H(\psi_i^2)$, on which $\ok$ acts by $\psi_i^1, \psi_i^2$, respectively. Now, the transitivity of the action of $\sigma$ on the $\psi_i^j$ means that we can order them so that

\

\renewcommand{\arraystretch}{1.3}
\begin{tabular}{l}
$\sigma \circ \psi_i^1 = \psi_{i+1}^1, \qquad i = 1, 2, \dots, f-1,$
\\ $\sigma \circ \psi_f^1 = \psi_1^2 $,
\\$\sigma \circ \psi_i^2 = \psi_{i+1}^2, \qquad i = 1, 2, \dots, f-1,$
\\$\sigma \circ \psi_f^2 = \psi_1^1$.
\end{tabular}

\

Let us choose a basis $\{e_i^j: i = 1, 2, \dots, f, j = 1, 2\}$ for $H$ such that $e_i^j$ spans $H(\psi_i^j)$. Note that the kernel of Frobenius on $\overline{H}:= H \pmod{p}$ is an $\ok$-module and is one dimensional in every $H(\varphi_i)$, because $\overline{A}$ satisfies the Rapoport condition, or, alternately, for each $i$, precisely one of $\{\psi_i^1, \psi_i^2\}$ belongs to the CM type. Suppose, without loss of generality, that $e_1^1$ spans the kernel of Frobenius in $\overline{H}(\varphi_1)$, then we must have that $\Fr(e_1^2)$, which is equal up to a unit to $e_2^2$, spans the kernel of Frobenius in $\overline{H}(\varphi_2)$ (this is where ``superspecial" is being used), and by the same rationale we find that the kernel of Frobenius in $\overline{H}(\varphi_i)$ is spanned by $e_i^1$, for $i$ odd, and by $e_i^2$, for $i$ even. In particular, the kernel of Frobenius in $\overline{H}(\varphi_f)$ is spanned by $e_f^2$, because $f$ is even. Now, by the same rationale, $\Fr(e_f^1)$ spans the kernel of Frobenius in $\overline{H}(\varphi_1)$, and it lies in $\overline{H}(\psi_1^2)$ because $\sigma \circ \psi_f^1 = \psi_1^2 $. This is a contradiction.
\end{proof}


\subsection{A description of $B_{p,L}$.}
Next we give a description of the quaternion algebra $B_{p,L}$ in terms of a CM field $K$, for a certain set of primes $p$, which according to Proposition~\ref{superspecialsplitting} includes the superspecial primes of $K$.  This description generalizes the approach of Gross and Zagier.

\vspace{5mm}

\noindent
{\bf Notation:}  If $\gerq$ is a prime of $L$, let $\alpha_\gerq$ denote a totally positive generator of $\gerq$.  It is unique up to an element of $\ol^{\times +}=\ol^{\times, 2}$.  Write $p = \prod_{\gerp \in S} \alpha_\gerp$. 

\begin{prop} \label{BpL} Let $K$ be a CM field, $K^+=L$. Assume $p$ is odd, unramified in $L$, and that all primes $\gerp \in S \setminus S_0$ split
in $K$ and all primes $\gerp \in S_0$ are inert in $K$.
These conditions imply that $K$ embeds in $B_{p,L}$.
Assume that the discriminant $\gerd_{K/L} = (d)$ satisfies $(d,2p)=1$.
Then there is a totally negative prime element
$\alpha_0 \in \ol$ such that $(\alpha_0,2pd)=1$ and $$B_{p,L} \cong \left(\frac{d, \, \alpha_0 p}{L}\right).$$  The ideal $(\alpha_0)$ is split in $K$.
\end{prop}

\begin{proof} We first need a lemma.
\begin{lem} (Primes in arithmetic progressions) \label{progressions}
Let $L$ be a number field and let $\nu_1, \dots, \nu_t$ be some of
$L$'s embeddings into $\RR$.  Let $\gerr \lhd \ol$ be an integral ideal and $r \in \ol$
an element such that $(r, \gerr) = 1$.  Then there is a prime element $\alpha \in \ol$ such that
$\alpha \equiv r \mod{\gerr}$ and $\nu_i(\alpha) > 0$, $\forall i=1,\dots,t$.
\end{lem}
\begin{proof}  We may assume $\nu_i(r) > 0$,  $\forall i=1,\dots,t$.  Indeed, one may replace
$r$ by $r+n$ for any element $n \in \gerr$.  Since $\gerr \otimes \QQ = L$,
for any $c \in \RR$, $\gerr$ contains elements $n$ such that $\nu(n) > c$ for every real place $\nu$
of $L$.  Taking $C=\max \{|\nu_i(r)|: \nu_i(r) < 0 \}$ and a suitable element $n \in \gerr$
we get $\nu_i(r+n) > 0$, $\forall i=1,\dots,t$.

Consider the modulus $\gerr \nu_1 \nu_2 \cdots \nu_t = \germ$, and the ray class group modulo $\germ$, $I(\germ)/P(\germ)$.
Here $I(\germ)$ is the multiplicative group of fractional ideals prime to $\germ$,
$P(\germ)$ is the subgroup of principal ideals having a generator $\beta$ such that
$\beta \equiv 1 \mod \germ$ and $\nu_i(\beta) >0$, $\forall i=1,\dots,t$.
Let $L(\germ)$ be the corresponding class field, $\Gal(L(\germ)/L) \cong I(\germ)/P(\germ)$.
The ideal $(r)$ is an element of $I(\germ)/P(\germ)$. Let $$\sigma = ((r), L(\germ)/L ) \in \Gal(L(\germ)/L)$$ be the Artin symbol.  By Chebotarev, there is a prime ideal $\gerp$ such that $(\gerp,\germ) = 1$ and $$\sigma = \sigma_\gerp = (\gerp, L(\germ)/L ).$$
Also, $\gerp$ is equivalent to $(r)$ modulo $P(\germ)$, hence also principal.
Indeed, $$\sigma_\gerp |_{H_L} = \sigma |_{H_L} = ((r), L(\germ)/L )|_{H_L} = 1.$$  Since
$\Gal(H_L/L) \cong I/P$, we must have that $\gerp$ is principal.  Let $(\alpha_1) = \gerp$.
By construction, $(\alpha_1) = (r)$  in $I(\germ)/P(\germ)$.  That means that the ideal
$(\alpha_1 r^{-1})$ has a generator $u\alpha_1 r^{-1}$, $u \in \ol^\times$, such that
$$u\alpha_1 r^{-1} \equiv 1 \mod \germ.$$  Let $\alpha = u\alpha_1 $.  Then $\alpha \equiv r \mod \germ$, meaning $\alpha \equiv r \mod \gerr$ and for every $i=1, \dots,t$, $\nu_i(\alpha)$
has the same sign as $\nu_i(r)$, i.e. is positive.
\end{proof}

According to Lemma~\ref{progressions}, we can choose $\alpha_0 \in \ol$ such that
\begin{enumerate}
\item $\alpha_0$ is a totally negative prime element of $\ol$.
\item $\alpha_0 \equiv p \mod{\eta^N}$, for each $\eta \mid 2$, some $N \gg 0$ (for choice of $N$, see below).
\item $\alpha_0 \equiv p \mod{\gerq}$, for each $\gerq \mid d$.
\item $\alpha_0 \equiv 1 \mod{p}$.
\end{enumerate}
Since $x^2 - d y^2 -\alpha_0 p z^2 \equiv 0 \mod{\eta^N}$ has a non-trivial solution if $N$ is large enough, then by Hensel's lemma, there is a $p$-adic solution .
We have
$$ (\star\star) \quad (d, \alpha_0 p)_\eta = 1 \, \,{\rm for \, \, all} \, \, \eta \mid 2, \quad \left(\frac{\alpha_0}{\gerq}\right) = \left(\frac{p}{\gerq}\right) \, \, {\rm for \, \, all} \, \,\gerq \mid d$$
and $(\alpha_0,2pd)=1$.

To show $B_{p,L} \cong \left(\frac{d, \, \alpha_0 p}{L}\right),$ we need to check:
\begin{enumerate}
\item For all $\eta \mid 2$, $(d,\alpha_0 p)_\eta = 1$: see $(\star\star)$.
\item For all $\eta$ finite such that $\eta \nmid d\alpha_0 p$, $(d,\alpha_0 p)_\eta = 1$: \\
$x^2 - d y^2 -\alpha_0 p z^2 \equiv 0 \mod{\eta}$ has a non-trivial solution.
\item For all $\eta$ finite such that $\eta \mid d$, $(d,\alpha_0 p)_\eta = 1$: \\
$x^2 -\alpha_0 p z^2 \equiv 0 \mod{\eta}$ has a non-trivial solution $ \iff (\frac{\alpha_0 p}{\eta})=1$, which is true by $(2)$.
\item For all $\eta \in S \setminus S_0$, $(d,\alpha_0 p)_\eta = 1$:\\
$x^2 - d y^2  \equiv 0 \mod{\eta}$ has a non-trivial solution $ \iff d = \square \mod{\eta} \iff \eta$ splits in $K$.
\item $\eta = \alpha_0 \Rightarrow (d,\alpha_0 p)_\eta = 1$:\\
$x^2 - d y^2 \equiv 0 \mod{\alpha_0}$ has a non-trivial solution $ \iff (\frac{d}{\alpha_0})=1$.
We will examine this below.
\item $\eta \in S_0 \Rightarrow (d,\alpha_0 p)_\eta = -1$:\\
$x^2 - d y^2  \equiv 0 \mod{\eta}$ has only the trivial solution $ \iff d \ne \square \mod{\eta} \iff \eta$ is inert in $K$.
\item $\eta$ real $ \Rightarrow (d,\alpha_0 p)_\eta = -1$:\\
$x^2 - d y^2 -\alpha_0 p z^2 = 0$ in $\RR$ has only the trivial solution since $-d$ and $-\alpha_0p$ are both positive.
\end{enumerate}
So it remains to prove only that $(\frac{d}{\alpha_0})=1$.

Write $d = (-1) \cdot \prod_{\gerq \mid d} \alpha_\gerq$, and $p = \prod_{\gerp \mid p} \alpha_\gerp$.
\begin{equation*}
\begin{split}
 \left( \frac{d}{\alpha_0} \right) & =  \left( \frac{-1}{\alpha_0}\right) \prod_{\gerq \mid d} \left( \frac{\alpha_\gerq}{\alpha_0}\right) \\
& = \left(\frac{-1}{\alpha_0}\right) \prod_{\gerq \mid d} \left(\left(\frac{\alpha_0}{\alpha_\gerq}\right) (\alpha_0,\alpha_\gerq)_2\right)
\quad {\rm (by \; quadratic \; reciprocity)}\\
& = \left(\frac{-1}{\alpha_0}\right)\prod_{\gerq \mid d} \left[\prod_{\gerp \mid p}\left(\frac{\alpha_\gerp}{\alpha_\gerq}\right)\right]
(\alpha_0,\alpha_\gerq)_2
\quad ({\rm since} \; \left(\frac{\alpha_0}{\gerq}\right) = \left(\frac{p}{\gerq}\right))\\
& = \left(\frac{-1}{\alpha_0}\right) (\alpha_0,-d)_2 \prod_{\gerq \mid d, \gerp \mid p}\left(\frac{\alpha_\gerp}{\alpha_\gerq}\right) \\
& = \left(\frac{-1}{\alpha_0}\right) (\alpha_0,-d)_2 \prod_{\gerq \mid d, \gerp \mid p}\left(\frac{\alpha_\gerq}{\alpha_\gerp}\right)
(\alpha_\gerp,\alpha_\gerq)_2
\quad {\rm(by \; quadratic \; reciprocity)}\\
& = \left(\frac{-1}{\alpha_0}\right) (\alpha_0,-d)_2 \prod_{\gerp \mid p}\left(\frac{-d}{\alpha_\gerp}\right)
(-d, \alpha_\gerp)_2 \\
& = \left(\frac{-1}{\alpha_0}\right) (\alpha_0,-1)_2 (\alpha_0,d)_2 \prod_{\gerp \mid p}\left(\frac{-1}{\alpha_\gerp}\right)
(\alpha_\gerp,-1)_2 (\alpha_\gerp,d)_2 \left(\frac{d}{\alpha_\gerp}\right) \\
& =(-1)^g (\alpha_0,d)_2 \prod_{\gerp \mid p} (\alpha_\gerp,d)_2 \left(\frac{d}{\alpha_\gerp}\right)
\quad { \rm (by \; Part \; (5) \; of \; Proposition \; \ref{QR} })\\
& =(-1)^g (\alpha_0 p,d)_2 (-1)^{\#S_0}
\quad {\rm (by \; our \; assumptions \; on \; } K)\\
& =(-1)^{g + \#S_0} \; ({\rm since} \; (\alpha_0 p, d)_\eta = 1, \forall  \eta \mid 2) \\
& = 1 \; ({\rm since} \; (g + \#S_0 ) = \# \, \{{\rm ramified \; primes \; of} \; B_{p,L} \} \;
{\rm is \; even}).\\
\end{split}
\end{equation*}
\end{proof}

\subsection{Another description of the quaternion algebra $B_{p,L}$}

\begin{dfn} For $\alpha$, $\beta \in \ok$ define $$[\alpha, \beta] :=  \begin{pmatrix} \alpha & \beta \\ \alpha_0 p \betabar & \alphabar \end{pmatrix}  \in M_2(K).$$
\end{dfn}

\begin{lem} Assumptions as in Proposition~\ref{BpL}.  $B_{p,L} \cong \{ [\alpha, \beta] \mid \alpha, \beta \in K \}.$
\end{lem}
\begin{proof} Proposition~\ref{BpL} implies that $B_{p,L} = L \bigoplus Li \bigoplus Lj \bigoplus Lij $, with $i^2 =d$, $j^2=\alpha_0p$, and $ij=-ji$.
We can write this as $K \bigoplus Kj$, with the multiplicative structure satisfying: for $x, y \in K$, $x(yj)=(xy)j$, $j^2=\alpha_0p$, and
$$ xj =(x_1+x_2i)j = x_1j+x_2ij = jx_1-jix_2 = j(x_1-ix_2)=j\overline{x}.
$$
So for the isomorphism $x+yj \rightarrow [x,y]$ to respect the multiplicative structure it is enough to check:
\begin{enumerate}
\item $[\alpha,0][0,\beta] = [0,\alpha\beta]$: $\begin{pmatrix}
\alpha & 0 \\ 0 & \alphabar \end{pmatrix} \begin{pmatrix} 0 & \beta \\ \alpha_0 p \betabar & 0 \end{pmatrix} =
\begin{pmatrix} 0 & \alpha  \beta \\ \alpha_0 p \overline{\alpha \beta} & 0 \end{pmatrix}. $

\item $[0,1]^2 = [\alpha_0 p,0]$:
$\begin{pmatrix}
0 & 1 \\ \alpha_0 p & 0 \end{pmatrix} \begin{pmatrix}
0 & 1 \\ \alpha_0 p & 0 \end{pmatrix} =
\begin{pmatrix} \alpha_0 p & 0  \\ 0 & \alpha_0 p  \end{pmatrix}. $

\item $[\alpha,0][0,1] = [0,1][\alphabar,0]$:
$\begin{pmatrix}
\alpha & 0 \\ 0 & \alphabar \end{pmatrix} \begin{pmatrix} 0 & 1 \\ \alpha_0 p & 0 \end{pmatrix} =
\begin{pmatrix}
0 & \alpha \\ \alpha_0 p \alphabar & 0 \end{pmatrix}=
\begin{pmatrix}
0 & 1 \\ \alpha_0 p  & 0 \end{pmatrix} \begin{pmatrix} \alphabar & 0 \\ 0 & \alpha \end{pmatrix}.$
\end{enumerate}
\end{proof}

\section{Orders in the quaternion algebra $B_{p,L}$} \label{orders}

By Proposition~\ref{BpL}, the ideal $\alpha_0\ol$ splits in $K$.
Write $$\alpha_0 \ok = \calA \cdot \overline{\calA},$$ and let $\calD = \calD_{K/L} = (\sqrt{d})$ be the different ideal of $K/L$.

\begin{dfn} \label{def:localsigns}
Let $\gera$ be an integral ideal of $\ok$.  For each $\gerq \mid d$, fix a solution $\lambda_\gerq$ to
\begin{equation} x^2 \equiv \alpha_0 p \mod \gerq.
\end{equation}
Let $\varepsilon(\gera, \gerq) \in \{ \pm 1 \}$ be a choice of sign for each $\gerq \mid d$.
Let $\lambda \in L$, $(\lambda,d)=1$, be such that
\begin{enumerate}
\item $\lambda \equiv \varepsilon(\gera, \gerq) \lambda_\gerq \mod \gerq$, $\forall \gerq \mid d$
\item $\lambda \calA^{-1} \gera^{-1}\overline{\gera}$ is an integral ideal of $\ok$.
\end{enumerate}
\end{dfn}
\noindent
This is possible by the Chinese Remainder Theorem and using that $(\calA^{-1}  \gera^{-1}\overline{\gera},d)=1$.

For example, one particular choice of signs which we will often make is  
$\varepsilon(\gera, \gerq) = (-1)^{\val_{\tilde{\gerq}}(\gera)}$, where
$\tilde{\gerq} \lhd \ok$ is an ideal such that $ \gerq \ok = \tilde{\gerq}^2$,
and we denote the corresponding $\lambda$ by $\lambda_\gera$.  
This will be explained further in Definitions~\ref{def:signsgera} and~\ref{def:signs}.

Let $\ell \in \ol$ be any non-zero element such that $(\ell, \alpha_0 d \gera^{-1}\overline{\gera})=1$ and $\ell$ is split in $K/L$.  In particular, $\ell$ could be a power of $p$.   Now define
$$ R = R(\gera, \lambda, \ell) = \{ [\alpha, \beta] \mid \alpha \in \calD^{-1}, \beta \in \calD^{-1}\calA^{-1} \ell \gera^{-1}\overline{\gera},
\alpha \equiv \lambda \beta \mod \ok \}. $$
\begin{prop} \label{Rorder} Assumptions as in Proposition~\ref{BpL}.  In particular, $K$ is a CM field such that $K^+=L$ has strict class number $1$, the discriminant of $K/L$ is prime to $2$ and thus square-free, and $p$ is odd, unramified in $K$. All primes $\gerp \in S \setminus S_0$ split in $K$ and all primes $\gerp \in S_0$ are inert in $K$. Then:
\begin{enumerate}
\item  $R$ is an order of $B_{p,L}$, containing $\ok$.
\item  $R$ has discriminant $p\cdot\ell$.
\item $R$ does not depend on the choice of $\lambda$, as long as $\lambda$ satisfies the same local sign conditions. 
\end{enumerate}
\end{prop}
\begin{proof}  Part (1). It is clear that $R$ is a finitely generated $\ol$-module, containing $\ok = \{[\alpha,0]: \alpha \in \ok \}$.  We need to show that $R$ is closed under multiplication.
The multiplication formula is $$[x,y][z,w] =[xz+\alpha_0 p y \overline{w}, xw+y\overline{z}],$$ and we need to show that, for
$[x,y]$, $[z,w] \in R$, also $[x,y][z,w] \in R$. \\

\noindent
{\emph {\bf Step 1.} Show that $xz+\alpha_0 p y \overline{w} \in \calD^{-1}$.}

\vspace{5mm}
A priori, $xz \in \calD^{-2}$,  and
$$\alpha_0 p y \overline{w} \in \alpha_0 p \calD^{-1}\calA^{-1} \ell \gera^{-1}\overline{\gera}
\overline{\calD^{-1}\calA^{-1}\ell \gera^{-1}\overline{\gera}} =
\alpha_0 p \calD^{-2}(\calA\overline{\calA})^{-1} \ell^2 = p \calD^{-2} \ell^2 \subseteq \calD^{-2}.$$
So it is enough to show: $\forall \tilde{\gerq} \mid \calD$, $\val_{\tilde{\gerq}}(xz+\alpha_0 p y \overline{w}) \ge -1$.
Let $\gerq = \tilde{\gerq} \cap \ol$.  Then $\gerq\ok = \tilde{\gerq}^2$.  We will work $\gerq$-adically. Let
$\pi \in \mathcal{O}_{K_{\tilde{\gerq}}}$ be a uniformizer such that $\pibar = -\pi$ (the extension of complex conjugation from $K$ to $K_{\tilde{\gerq}}$).
\begin{lem} Such a $\pi$ exists.
\end{lem}
\begin{proof}  Choose a uniformizer $\pi_0$ of $\mathcal{O}_{L_{\gerq}}$, and let $K_1 = L_{\gerq}(\sqrt{\pi_0})$.
Then for $K_1$ there exists such a uniformizer.  So it is enough to show that, if $\gerq \mid q$ and $ q \ne 2$, then any $q$-adic
field $L_1$ has a unique quadratic ramified extension.  By Local Class Field Theory, ramified quadratic extensions are in bijection with subgroups
of index $2$ of $\mathcal{O}_{L_{1}}^{\times}$.
There is a unique subgroup of index $2$ of $\mathcal{O}_{L_{1}}^{\times}$ since it contains $\mathcal{O}_{L_{1}}^{\times 2}$ and
$\mathcal{O}_{L_{1}}^{\times}/\mathcal{O}_{L_{1}}^{\times 2} \cong \ZZ/2\ZZ$.
\end{proof}

Note that  $\calD^{-1}\calA^{-1}\ell \gera^{-1}\overline{\gera} \mathcal{O}_{K_{\tilde{\gerq}}} = \frac{1}{\pi}\mathcal{O}_{K_{\tilde{\gerq}}}$, since $({\calA},\tilde{\gerq})=1$,  $(\ell,\tilde{\gerq})=1$ and $(\gera^{-1}\overline{\gera},\tilde{\gerq})=1$ because $\gera^{-1}\overline{\gera}$ has no ramified or inert primes.
Write then $x=\frac{x_0}{\pi}$, $y=\frac{y_0}{\pi}$, $z=\frac{z_0}{\pi}$, $w=\frac{w_0}{\pi}$,
with $x_0$, $y_0$, $z_0$, $w_0 \in \mathcal{O}_{K_{\tilde{\gerq}}}$.
So
$$x \equiv \lambda y \mod{\ok} \Rightarrow x_0 - \lambda y_0 \in (\pi)$$ and
$$z \equiv \lambda w \mod{\ok} \Rightarrow z_0 - \lambda w_0 \in (\pi).$$
Now
$$ xz + \alpha_0 p y \overline{w} = \frac{1}{\pi^2}(x_0 z_0 - \alpha_0 p y_0 \overline{w_0}),$$
so it is enough to show:
$\val_{\tilde{\gerq}}(x_0 z_0 -\alpha_0 p y_0 \overline{w_0}) \ge 1$. But

\begin{equation*}
\begin{split}
x_0 z_0 - \alpha_0 p y_0 \overline{w_0}
& \equiv \lambda y_0 \lambda w_0 - \alpha_0 p y_0 \overline{w_0} \mod{(\pi)} \\
& \equiv \lambda^2 y_0  w_0 -\alpha_0 p y_0 w_0 \mod{(\pi)}, \;
({\rm because \; conjugation \; is \; trivial}\mod{(\pi)}) \\
& \equiv (\lambda^2 -\alpha_0 p) y_0 w_0 \\
& \equiv  (\lambda_{\gerq}^2 -\alpha_0 p) y_0 w_0 \\
& \equiv 0 \mod{(\pi)}.
\end{split}
\end{equation*}

\vspace{5mm}
\noindent
{\emph {\bf Step 2.} Show that $xw + y\overline{z} \in \calD^{-1}\calA^{-1} \ell \gera^{-1}\overline{\gera}$.}

\vspace{5mm}
A priori, $xw$ and
$y\overline{z} \in \calD^{-2}\calA^{-1} \ell \gera^{-1}\overline{\gera}$, so
we just need to show $\val_{\tilde{\gerq}}(x w + y  \overline{z}) \ge -1$ at all primes
$\tilde{\gerq} \mid \calD$.  Using the same notation as in step 1, we need to show
$\val_{\tilde{\gerq}}(x_0 w_0 - y_0 \overline{z_0}) \ge 1$.
We have, modulo $(\pi)$: $x_0 w_0 - y_0 \overline{z_0} = x_0 w_0 - y_0 z_0 = \lambda y_0 w_0 - \lambda y_0 w_0 =0$.

\vspace{5mm}
\noindent
{\emph {\bf Step 3.}
Show that $xz+\alpha_0 p y \overline{w} - \lambda (x w + y \overline{z}) \in \ok$.}

\vspace{5mm}
A priori, by Steps 1 and 2, $xz+\alpha_0 p y \overline{w} \in \calD^{-1}$ and
$$\lambda (x w + y \overline{z}) \in \calD^{-1} \ell \lambda \calA^{-1} \gera^{-1}\overline{\gera} \subset \calD^{-1} \ell \subset \calD^{-1},$$ since
$\lambda \calA^{-1} \gera^{-1}\overline{\gera} \subseteq \ok $.
Therefore, we just need to show that for all $\tilde{\gerq} \mid \calD$,  $$\val_{\tilde{\gerq}}(xz+\alpha_0 p y \overline{w} - \lambda (x w + y \overline{z})) \ge 0.$$
Using the same notation as above, this is equivalent to:
$$\val_{\tilde{\gerq}}(x_0 z_0 -\alpha_0 p y_0 \overline{w_0} - \lambda (x_0 w_0 - y_0 \overline{z_0})) \ge 2.$$
Write $x_0 = \lambda y_0 + \pi x_1$ and $z_0 = \lambda w_0 + \pi z_1$.  Then
\begin{equation*}
\begin{split}
(\lambda y_0 + \pi x_1)(\lambda w_0 + \pi z_1) - \alpha_0 p y_0 \overline{w_0} - \lambda(\lambda y_0 + \pi x_1)w_0 + \lambda y_0 (\lambda \overline{w_0} - \pi \overline{z_1}) \\
= (\lambda^2-\alpha_0 p)y_0 \overline{w_0} + \lambda \pi y_0 (z_1 - \overline{z_1})
\equiv 0 \mod{\pi^2},
\end{split}
\end{equation*}
since $(z_1 - \overline{z_1}) \in (\pi)$ and $(\lambda^2-\alpha_0 p) \in \gerq\mathcal{O}_{L_\gerq} \subset (\pi^2).$


\vspace{5mm}
\noindent
Part (2).
We need to compute the discriminant of
$$ R = R(\gera, \lambda, \ell) = \{ [\alpha, \beta] \mid \alpha \in \calD^{-1}, \beta \in \calD^{-1}\calA^{-1} \ell \gera^{-1}\overline{\gera},
\alpha \equiv \lambda \beta \mod \ok \}. $$
Let
$$ R' = \{ [\alpha, \beta] \mid \alpha \in \ok, \beta \in  \ell \gera^{-1}\overline{\gera} \}. $$
$R'$ is an $\ol$-module of rank $4$.

\begin{lem} $\disc(R') = (\ell \alpha_0 p d)^2$
\end{lem}
\begin{proof}
The quadratic form on $R'$ is $\det [\alpha, \beta] = \alpha\alphabar - \alpha_0 p \beta \betabar =: q([\alpha, \beta])$.
Note that this quadratic form coincides with the norm form on the the quaternion algebra $B_{p,L}$: writing
$$[\alpha, \beta] = [\alpha, 0] + [0,\beta][0,1] = (\alpha_1 + \alpha_2 i) + (\beta_1 + \beta_2 i)j,
$$ where $i^2 =d$ and $j^2 = \alpha_0 p$, we have
\begin{equation*}
\begin{split}
\Norm((\alpha_1 + \alpha_2 i + \beta_1 j + \beta_2 ij) 
& = \alpha_1^2 - \alpha_2^2 d - \beta_1^2 \alpha_0 p + \beta_2^2 d \alpha_0 p \\
& = (\alpha_1 + \alpha_2 i)(\alpha_1 - \alpha_2 i)  - \alpha_0 p (\beta_1 + \beta_2 i)(\beta_1 - \beta_2 i) \\
& = \alpha \alphabar - \alpha_0 p \beta \betabar.
\end{split}
\end{equation*}
The associated bilinear form is
$$ \langle [\alpha, \beta], [\gamma, \delta] \rangle
= \alpha \gammabar + \alphabar \gamma - \alpha_0 p (\beta \deltabar + \betabar \delta),
$$ where $\frac{1}{2}\langle x,x\rangle = q(x)$.
Note that
$ \langle [\alpha, 0], [0, \delta] \rangle = 0$ and
$$ \langle [\alpha_1, 0], [\alpha_2, 0] \rangle = \alpha_1 \alphabar_2 + \alphabar_1 \alpha_2 =  \Tr_{K/L} \alpha_1 \alphabar_2,$$
$$ \langle [0, \beta_1], [0,\beta_2] \rangle = - \alpha_0 p (\beta_1 \betabar_2 + \betabar_1 \beta_2) =  -\alpha_0 p \Tr_{K/L} \beta_1 \betabar_2.$$
To compute the discriminant of $R'$ with respect to the bilinear form, we need to compute the determinant of the matrix
$ \left( \langle x_i, x_j \rangle \right)$, for $\{x_i\}$ a basis for $R'$.
Choose a basis $\{ w_1, w_2\}$ for $\ok$ as an $\ol$-module (e.g. $\{ 1, t\}$).  Choose a basis $\{ w_3, w_4\}$ for $\ell \gera^{-1}\overline{\gera}$ as an $\ol$-module.
By the above calculations, we see that $$\det\left( \langle w_i, w_j \rangle \right) = \det(M_1)\det(M_2),$$ where
$$ M_1 = \begin{pmatrix}
2w_1 \overline{w_1} & w_1 \overline{w_2} + w_2 \overline{w_1} \\  w_1 \overline{w_2} + w_2 \overline{w_1} & 2 w_2 \overline{w_2} \end{pmatrix}
= (\Tr(w_i \overline{w_j})),$$ $i,j = 1,2$ and
$$ M_2 = -\alpha_0 p
\begin{pmatrix}
2w_3 \overline{w_3} & w_3 \overline{w_4} + w_4 \overline{w_3} \\  w_3 \overline{w_4} + w_4 \overline{w_3} & 2 w_4 \overline{w_2} \end{pmatrix}
= -\alpha_0 p (\Tr(w_i \overline{w_j})),$$ $i,j = 3,4$.
We have $$\det(M_1) = -\disc_{K/L}(\ok) $$ and
$$\det(M_2) = -(\alpha_0 p)\disc_{K/L}(\ell \gera^{-1}\overline{\gera}). $$
For any $\ok$-ideal $\gerb$, $\disc_{K/L}(\gerb) = \disc_{K/L}(\ok) \Norm_{K/L}(\gerb)^2$~\cite[Prop. 13, p. 66]{Lang}.
So $$\disc(R') = \disc_{K/L}(\ok)^2 \Norm_{K/L}(\ell \gera^{-1}\overline{\gera})^2 (\alpha_0 p)^2 = (\ell \alpha_0 p d)^2.$$

Remark: This uses that $\ell$ is split in $K/L$. In a typical application, $\ell$ will be a prime lying above $p$.  If $p$ is inert in $L$, then it will automatically be split in $K/L$ according to the hypotheses of Proposition~\ref{BpL}.
If $\ell$ is not split in $K/L$, we get a higher power of $\ell$ in the final answer.

\end{proof}
In order to show that $R$ has discriminant $p\cdot\ell$ the following lemma is needed:

\begin{lem} The following sequence is exact:
$$ 0 \rightarrow R' \injects R \rightarrow^{\psi} \calD^{-1}\calA^{-1}/\ok \rightarrow 0,$$
where
$$  [\alpha,\beta] \mapsto \beta \in
\frac{\calD^{-1}\calA^{-1} \ell \gera^{-1}\overline{\gera}}
{\ell \gera^{-1}\overline{\gera}}
\cong \calD^{-1}\calA^{-1}/\ok. $$
\end{lem}
\begin{proof}
First $R' \subseteq R$ because $\alpha \in \ok$, $\lambda \beta \in \lambda \ell \gera^{-1}\overline{\gera} = (\lambda \gera^{-1}\overline{\gera})\ell \subseteq \ok \ell \subseteq \ok$.
Since $\lambda \beta \in \ok$, clearly $\alpha \equiv \lambda \beta \pmod{\ok}$.
Now:
\begin{itemize}
\item {\bf Exactness at $R$:} $R' \subseteq \Ker(\psi)$ is clear.  Now suppose
$[\alpha, \beta] \in \Ker(\psi)$.  Then $\beta \in \ell \gera^{-1}\overline{\gera}$ and so $\alpha \in \ok$ because $\lambda \beta \in \ok$, by the definition of $\lambda$, and $\alpha \equiv \lambda \beta \pmod{\ok}$.  So
$[\alpha,\beta] \in R'$.

\item
{\bf $\psi$ surjective:}  Let $\beta \in \calD^{-1}\calA^{-1} \ell \gera^{-1}\overline{\gera}$.  Then $[\lambda \beta, \beta] \in R$ because
$\lambda \beta \in \calD^{-1}\ell (\lambda \calA^{-1}  \gera^{-1}\overline{\gera})
\subseteq \calD^{-1}\ell \ok \subseteq \calD^{-1}$.
\end{itemize}
\end{proof}
Thus $ \disc_{K/L} (R) = \disc_{K/L}(R')/\Norm_{K/L}(\calD\calA)^2
= (\ell \alpha_0 p d)^2/(\alpha_0 d)^2 = \ell^2 p^2$, so the discriminant of
$R$ as an order of $B_{p,L}$ is $\ell p$.
\end{proof}

\noindent
Part(3). Finally, $R$ is independent of the choice of $\lambda$ assuming $\lambda$ satisfies the same local sign conditions: 
\begin{proof}
Suppose both $\lambda$ and $\lambda'$ satisfy the conditions of Definition~\ref{def:localsigns}.
Let $[\alpha,\beta] \in R(\gera, \lambda, \ell)$, so $\alpha \in \calD^{-1}$,
$\beta \in \calD^{-1}\calA^{-1} \ell \gera^{-1}\overline{\gera}$, and
$\alpha \equiv \lambda \beta \mod \ok$.
Then, $$\alpha - \lambda \beta \in \ok \implies (\sqrt{d}\alpha) - \lambda (\sqrt{d}\beta) \in (\sqrt{d}),$$ and
$$(\sqrt{d}\alpha) - \lambda' (\sqrt{d}\beta) -(\lambda-\lambda') (\sqrt{d}\beta) \in (\sqrt{d}).$$
Now, because $d$ is square free and for all $\gerq | d$ we have $\lambda' = e(\gera, \gerq)\lambda_\gerq = \lambda  \pmod{\gerq$}, it follows that $\lambda - \lambda' \in (d)$. But,
$$ \lambda-\lambda' \in (d) \implies (\lambda-\lambda')\sqrt{d}\beta
\in d \ell \calA^{-1} \gera^{-1}\overline{\gera},$$
and
$$ \lambda \sqrt{d} \beta - \lambda' \sqrt{d} \beta \in \ok$$
by the definition of $\lambda$ and $\lambda'$, so
$$ (\lambda - \lambda')\sqrt{d} \beta \in \ok \cap d \ell \calA^{-1} \gera^{-1}\overline{\gera} \subseteq (d).$$
It follows that $(\sqrt{d}\alpha) - \lambda' (\sqrt{d}\beta) \in (\sqrt{d}),$ so
$\alpha \equiv \lambda' \beta \mod \ok$.
\end{proof}

\section{Classification of superspecial orders of $B_{p,L}$ in which $\ok$ embeds,
having chosen an embedding $K \hookrightarrow B_{p,L}$} \label{superspecialorders}

By a superspecial order in $B_{p, L}$ we mean an order of discriminant $p\calO_L$. An example of such an order is $R\otimes_\ZZ \calO_L$ for a maximal order $R$ of $B_{p, \infty}$. Let $K$ be a primitive CM field such that $K^+=L$. As before, we denote by $d$ a totally negative generator of the relative different ideal $\calD_{K/L}$. In this section we classify the superspecial orders in which $\calO_K$ embeds, relying on the results in the previous section and making the particular choice of local signs 
$\varepsilon(\gera, \gerq) = (-1)^{\val_{\tilde{\gerq}}(\gera)}$, where
$\tilde{\gerq} \lhd \ok$ is an ideal such that $ \gerq \ok = \tilde{\gerq}^2$,
and we denote the corresponding $\lambda$ by $\lambda_\gera$.
Our classification of these orders will be achieved through the following series of lemmas.

\begin{lem}
Let $R_1$, $R_2$ be two superspecial orders in $B_{p, L}$.  Then $R_1 \cong R_2$ over $K \iff
\exists \mu \in K$ such that $R_1 = \mu R_2 \mu^{-1}$.
\end{lem}
\begin{proof}
By Skolem-Noether,
$R_1 \cong R_2 \iff \exists \mu \in B_{p,L}^{\times}$ such that
$R_1 = \mu R_2 \mu^{-1}$.
This is a $K$-automorphism if and only if $\mu \in \Cent_{B_{p,L}}(K) = K$.
\end{proof}

We make the following choice of local signs for $\lambda$ and introduce the notation $\lambda_\gera$.
\begin{dfn} \label{def:signsgera}
For $\gera \lhd \ok$, let $\lambda_\gera := \lambda_{\varepsilon(\gera)},$
where $\varepsilon(\gera, \gerq) = (-1)^{\val_{\tilde{\gerq}}(\gera)}$, and 
$\tilde{\gerq} \lhd \ok$ is an ideal such that $ \gerq \ok = \tilde{\gerq}^2$.
\end{dfn}

\begin{lem} \label{lem:ramifiedideal} Given $\gera$, $\lambda = \lambda_\gera$ as in Definition~\ref{def:signsgera}, there exists $\gerc \mid d$ such that $R(\gera, \lambda) = R(\gera\gerc, \lambda_{\gera\gerc})$.
\end{lem}

\begin{proof} $R(\gera\gerc, \lambda_{\gera\gerc}, \ell) = R(\gera, \lambda_{\gera}\cdot\lambda_{(-1)^{\val_{\tilde{\gerq}}(\gerc)}} , \ell)$
because $$\lambda_{\gera\gerc} \equiv (-1)^{\val_{\tilde{\gerq}}(\gera\gerc)}
\lambda_\gerq \mod \gerq, \forall \gerq \mid d,$$ so
$$\lambda_{\gera\gerc} \equiv \lambda_{\gera}(-1)^{\val_{\tilde{\gerq}}(\gerc)}
 \mod \gerq, \forall \gerq \mid d.$$
So as $\gerc$ ranges over the ideals dividing $d$, we get all sign vectors
$\varepsilon(\gera)$ that appear in the left hand side, and each one once.
\end{proof}

\begin{lem} \label{lem:representatives} Fix $\{\gerb_1, \dots,\gerb_h \}$ representatives for the class group
of $K$ and the choice of local signs as above.  Then every $R(\gera, \lambda_\gera)$ is isomorphic to $R(\gerb, \lambda_\gerb)$
for some $\gerb \in \{\gerb_1, \dots,\gerb_h \}$.
\end{lem}
\begin{proof} Let $\mu \in K^{\times}$ be such that $\gerb = \mu \gera$ for some (unique) $\gerb \in \{\gerb_1, \dots,\gerb_h \}$.

$$ \mu^{-1} R(\gera, \lambda_\gera) \mu = \left\{
\begin{pmatrix} \mu^{-1} & 0   \\ 0        & \mubar^{-1}   \end{pmatrix}
\begin{pmatrix} \alpha & \beta \\ \alpha_0 p \betabar & \alphabar \end{pmatrix}
\begin{pmatrix}  \mu & 0        \\ 0        & \mubar   \end{pmatrix}
: \alpha \in \calD^{-1}, \beta \in \calD^{-1}\calA^{-1} \gera^{-1}\overline{\gera},
\alpha \equiv \lambda_\gera \beta \mod \ok \right\}.
$$
$$ = \left\{
\begin{pmatrix} \alpha & \frac{\mubar}{\mu}\beta \\
\alpha_0 p \overline{(\frac{\mubar}{\mu}\beta)}  & \alphabar \end{pmatrix}
: \alpha \in \calD^{-1}, \beta \in \calD^{-1}\calA^{-1} \gera^{-1}\overline{\gera},
\alpha \equiv \lambda_\gera \beta \mod \ok \right\}
$$
by setting $b= \frac{\mubar}{\mu}\beta$, this is equal to
$$ = \left\{
\begin{pmatrix} \alpha & b \\
\alpha_0 p \overline{b}  & \alphabar \end{pmatrix}
: \alpha \in \calD^{-1}, b \in \calD^{-1}\calA^{-1} \gerb^{-1}\overline{\gerb},
\alpha \equiv \lambda_\gera \frac{\mu}{\mubar} b \mod \ok \right\}
$$
because $\gerb = \mu\gera$,
$$\frac{\mubar}{\mu} \beta \in \calD^{-1}\calA^{-1}\gera^{-1}\overline{\gera}\frac{\mubar}{\mu}
= \calD^{-1}\calA^{-1} \gerb^{-1}\overline{\gerb},$$
and $\alpha \equiv \lambda_\gera \frac{\mu}{\mubar} \frac{\mubar}{\mu} \beta
= \lambda_\gera \frac{\mu}{\mubar} b \mod \ok .$

Now it remains to show that $\alpha \equiv \lambda_\gera \frac{\mu}{\mubar} b \mod \ok \iff
\alpha \equiv \lambda_\gerb b \mod \ok .$  Equivalently,
$$(\sqrt{d} \alpha) \equiv \lambda_\gera \frac{\mu}{\mubar} (\sqrt{d}b) \mod \tilde{\gerq}, \,
\forall \tilde{\gerq}\mid\sqrt{d}\ok
\iff
(\sqrt{d} \alpha) \equiv \lambda_\gerb (\sqrt{d}b) \mod \tilde{\gerq}, \,
\forall \tilde{\gerq}\mid\sqrt{d}\ok.
$$
This can be checked in $\calO_{K_{\tilde{\gerq}}}$ for every $\tilde{\gerq}$.
The point is that
$ (-1)^{\val_{\tilde{\gerq}}(\gerb)} = (-1)^{\val_{\tilde{\gerq}}(\gera)}\cdot(-1)^{\val_{\tilde{\gerq}}(\mu)}$,
and so it is enough to show that
$ \frac{\mu}{\mubar} \equiv (-1)^{\val_{\tilde{\gerq}}(\mu)} \mod \tilde{\gerq}.$
This follows from the fact that $\calO_{K_{\tilde{\gerq}}} = \calO_{L_{{\gerq}}}[\pi]$,
with $\pibar = -\pi$, so writing $\mu = \pi^r\cdot u$, $u \in \calO_{K_{\tilde{\gerq}}}^{\times}$,
we have $\overline{u} = u \mod \tilde{\gerq}$, and
$$ \frac{\mu}{\mubar} = (-1)^{r}\frac{u}{\overline{u}} \equiv (-1)^r \mod \tilde{\gerq}.$$
Thus, we have proved that $\mu^{-1} R(\gera, \lambda_\gera) \mu =  R(\mu\gera, \lambda_{\mu\gera}).$
\end{proof}

\begin{lem} \label{lem:equivorders} $R(\gera, \lambda_\gera) = R(\gerb, \lambda_\gerb) \iff
\gera^{-1}\overline{\gera} = \gerb^{-1}\overline{\gerb}$ and
$\forall \tilde{\gerq} \mid d$, $\val_{\tilde{\gerq}}(\gera) \equiv \val_{\tilde{\gerq}}(\gerb) \mod{2}$.
\end{lem}
\begin{proof} $\Leftarrow$ obvious.

($\Rightarrow$) Let $\beta \in \calD^{-1}\calA^{-1} \gera^{-1}\overline{\gera}$ and let
$\alpha = \lambda_\gera \beta$.  Since $\lambda_\gera \calA^{-1} \gera^{-1}\overline{\gera}
\subseteq \ok$, it follows that $\alpha \in \calD^{-1}$.
Therefore $[\alpha,\beta] \in R(\gera, \lambda_\gera) = R(\gerb, \lambda_\gerb)$ and so
$\beta \in \calD^{-1}\calA^{-1} \gerb^{-1}\overline{\gerb}$.  Therefore
$ \calD^{-1}\calA^{-1} \gera^{-1}\overline{\gera} \subseteq \calD^{-1}\calA^{-1} \gerb^{-1}\overline{\gerb}$.  By symmetry we have equality.

Furthermore, since $[\lambda_\gera \beta, \beta] \in R(\gerb, \lambda_\gerb)$, we have
$$ \lambda_\gera \beta \equiv \lambda_\gerb \beta \mod \ok, \,
\forall \beta \in \calD^{-1}\calA^{-1} \gera^{-1}\overline{\gera}.$$
Otherwise said,
$$   \beta(\lambda_\gera - \lambda_\gerb) \equiv 0 \mod \ok, \,
\forall \beta \in \calD^{-1}\calA^{-1} \gera^{-1}\overline{\gera},$$
and this implies
$$  \lambda_\gera \equiv \lambda_\gerb \mod \calD^{-1}\calA^{-1} \gera^{-1}\overline{\gera}. $$
We conclude that
$$  \forall \tilde{\gerq} \mid d, \, \lambda_\gera \equiv \lambda_\gerb \mod \tilde{\gerq},
\; {\rm because } \,(\calD, \calA\gera \overline{\gera}^{-1})=1).$$
It follows that
$$   \forall \tilde{\gerq} \mid d, \,
(-1)^{\val_{\tilde{\gerq}}(\gera)} = (-1)^{\val_{\tilde{\gerq}}(\gerb)}  .$$
\end{proof}

\begin{lem} For $\gerb$, $\gerb' \in \{\gerb_1, \dots,\gerb_h \}$, 
$R(\gerb, \lambda_\gerb) \sim R(\gerb', \lambda_{\gerb'}) \iff \gerb = \gerb'$.
\end{lem}
\begin{proof} $\Leftarrow$ obvious.

Suppose $R(\gerb, \lambda_\gerb) = \mu^{-1} R(\gerb', \lambda_{\gerb'}) \mu =
R(\mu\gerb', \lambda_{\mu\gerb'})$,
(this second equality was proved in Lemma~\ref{lem:representatives} above).
By Lemma~\ref{lem:equivorders}, this implies that
$$\gerb^{-1}\overline{\gerb} =\gerb'^{-1}\overline{\gerb'}\frac{\mubar}{\mu}$$
or
$$\gerb'\gerb^{-1}\mu =\overline{\gerb'\gerb^{-1}\mu}.$$
An ideal $\gerf \lhd \ok$ satisfies $\gerf = \overline{\gerf}$ if and only if
$\gerf = j \cdot \prod_{\tilde{\gerq} \mid d} \tilde{\gerq}^{s(\tilde{\gerq})}$, $j \in L$.
Indeed, write $\gerf$ as a product of inert, split, and ramified prime ideals.  Inert prime
ideals are generated by elements of $L$.  Split prime ideals must appear in the factorization
to the same power as their complex conjugate, because of the condition $\gerf = \overline{\gerf}$.
Thus it is actually some power of their norm which appears, and that is also generated by
an element of $L$.  What remains is a product of some ramified primes.

Applying this to the ideal $\gerf = \gerb'\gerb^{-1}\mu$, we find that
$$\mu \gerb'= j \cdot \prod_{\tilde{\gerq} \mid d} \tilde{\gerq}^{s(\tilde{\gerq})}\cdot \gerb.$$
Note that $R(\mu\gerb', \lambda_{\mu\gerb'}) = R(\frac{\mu}{j}\gerb', \lambda_{\frac{\mu}{j}\gerb'}),$
so we can replace $\mu$ by $\mu/j$ to obtain
$R(\gerb, \lambda_\gerb) = R(\mu\gerb', \lambda_{\mu\gerb'})$ with $\mu\gerb'$ of the form
$$\mu \gerb'=  \prod_{\tilde{\gerq} \mid d} \tilde{\gerq}^{s(\tilde{\gerq})}\cdot \gerb.$$
Now $\lambda_{\gerb} = \lambda_{\mu\gerb'}$ implies that each $s(\tilde{\gerq})$ is even, so
$\mu\gerb' = k\gerb$ for some $k \in K$.  Thus $\gerb' = \gerb$ because they are already representatives for the class group.
\end{proof}

\begin{lem} \label{lem:ordersRa} Any superspecial order $R \supseteq \ok$ is isomorphic to some $R(\gera, \lambda)$.
\end{lem}
\begin{proof} Let $\gerc$ be a prime ideal of $L$.  For any ideal $\gera$ of $K_\gerc$,
define orders $R^\gerc(\gera, \lambda_\gera)$ of $(B_{p,L})_\gerc$ exactly the same way as
for $R(\gera, \lambda_\gera)$.
The orders have the same properties that were proved for the $R(\gera, \lambda_\gera)$ in Proposition~\ref{Rorder}: independent of the choice of $\lambda$, conductor $p\calO_{L_\gerc}$.

Then, for an ideal $\gera$ of $K$ we have
$R(\gera, \lambda_\gera)_\gerc = R^\gerc(\gera_\gerc, \lambda_{\gera_\gerc})$.
Let $R$ be an order of $B_{p,L}$ that contains $\ok$, of discriminant $p\ol$.
For every $\gerc$, the order $R_\gerc$ is an Eichler order of discriminant $p\calO_{L_\gerc}$, as is
the order $R(\calO, \lambda_\calO)_\gerc$, where $\calO$ represents the trivial ideal class.
For every $\gerc$ there is a $\mu_\gerc \in (B_{p,L})_\gerc^\times$ such that
$$R_\gerc = \mu_\gerc^{-1} R(\calO, \lambda_\calO)_\gerc \mu_\gerc,$$
because Eichler orders of the same discriminant are locally conjugate.
Furthermore, for almost all $\gerc$,
$$R_\gerc = \M_2(\calO_{L_\gerc}) \subseteq (B_{p,L})_\gerc = \M_2({L_\gerc}),$$
and the same holds for $R(\gera,\lambda_\gera)$.
Now it is enough to show that we can choose $\mu_\gerc \in K_\gerc^\times$ for all $\gerc$,
because in that case
$$R_\gerc = \mu_\gerc^{-1} R(\calO, \lambda_\calO)_\gerc \mu_\gerc
= R^\gerc((\mu_\gerc), \lambda_{(\mu_\gerc)}),$$
for a collection of elements
$$\left\{ \mu_\gerc : \gerc \lhd \ol \;{\rm prime } \;, \mu_\gerc = 1
\;{\rm for  } \;{\rm almost  } \;{\rm all } \; \gerc,
\mu_\gerc \in K_\gerc^\times \right\}.$$
Therefore, there is an ideal $\gera$ of $K$ such that, for all $\gerc$,
$\gera_\gerc = (\mu_\gerc)$.  The two orders
$R$ and $R(\gera, \lambda_\gera)$ are equal because they are equal locally everywhere, and we are done.

To show that we may choose $\mu_\gerc \in K_\gerc^\times$ for all $\gerc$,
we use~\cite[Theorems 3.1, 3.2, pp. 43-44]{Vigneras}, to produce an element
$\nu_\gerc$ such that

(i) $ \nu_\gerc^{-1}(\mu_\gerc^{-1} R(\calO, \lambda_\calO)_\gerc \mu_\gerc) \nu_\gerc
= \mu_\gerc^{-1} R(\calO, \lambda_\calO)_\gerc \mu_\gerc = R_\gerc $, 
and 

(ii) the embedding of $\calO_{K_\gerc}$ into $R_\gerc$ is the embedding of
$\calO_{K_\gerc}$ into $R(\calO, \lambda_\calO)_\gerc$ conjugated by $\nu_\gerc \mu_\gerc$.

\noindent
Since conjugation by $\nu_\gerc \mu_\gerc$ fixes $K_\gerc$ pointwise, this implies
$\nu_\gerc \mu_\gerc$ commutes with $K_\gerc$, and so
$\nu_\gerc \mu_\gerc \in K_\gerc^\times$.
\end{proof}

Our conclusion is that isomorphism classes of superspecial orders of $B_{p,L}$ in which $\ok$ embeds are the isomorphism classes of $R(\gera, \lambda_\gera)$.  Thus we have proved the following theorem:
\begin{thm} \label{classification} Fix an embedding of $K \hookrightarrow B_{p,L}$.  The isomorphism classes of the superspecial orders in which $\ok$ embeds are in bijection with the ideal class group of $K$ via the map $$ [\gera] \mapsto R(\gera, \lambda_\gera).$$
\end{thm}

\begin{rmk} In the case $L=\QQ$, Theorem~\ref{classification} provides a different proof for the main theorems of Dorman's paper on global orders in definite quaternion algebras~\cite{DormanOrders} and corrects several errors and gaps in the proofs there.
For example, we correct the missing condition on the integrality for
$\lambda \calD^{-1}\calA^{-1} \gera^{-1}\overline{\gera} $ and the
resulting mistake in proof of Proposition 2, and we give a different proof of the
1-1 correspondence.
\end{rmk}

\section{Main theorems on counting formulas} \label{maincounting}

\subsection{Assumptions and notation.} \label{assumptions}
Let $L$ be a totally real field of degree $g$ of strict class number one, $p$ a rational prime which is unramified in $L$, and $K$ a primitive CM field with $K^+=L$.
Using the same notation as in Lemma~\ref{Lemma: CMdisc}, write the ring of integers of $K$, $\ok = \ol[t]$, where $t^2+at+b=0$ for some $a,b \in \ol$, and the different $\calD = \calD_{K/L} = ({\sqrt{d}})$, with $d=a^2-4b$ a totally negative element of $\ol$.

Assume as in Proposition~\ref{BpL} that all primes $\gerp \in S \setminus S_0$ split in $K$ and all primes $\gerp \in S_0$ are inert in $K$ and that the discriminant $\gerd_{K/L} = (d)$ satisfies $(d,2)=1$ and $(d,p)=1$.
Let $\alpha_0 \in \ol$ be a totally negative prime element such that $$B_{p,L} \cong \left(\frac{d, \, \alpha_0 p}{L}\right),$$
where $(\alpha_0,2pd)=1$, $\alpha_0 \equiv p \mod{\gerq}$ for each $\gerq \mid d$,
$\alpha_0 \equiv 1 \mod{p}$, and $\alpha_0 \ok = \calA \cdot \overline{\calA}$.

Let $\gera$ be an integral ideal of $\ok$.  For each $\gerq \mid d$, fix a solution $\lambda_\gerq$ to
\begin{equation} x^2 \equiv \alpha_0 p \mod \gerq.
\end{equation}
Let $\varepsilon(\gera, \gerq) \in \{ \pm 1 \}$ be a choice of sign $\forall \gerq \mid d$.
Let $\lambda \in \ol$ be such that
\begin{enumerate}
\item $\lambda \equiv \varepsilon(\gera, \gerq) \lambda_\gerq \mod \gerq$, $\forall \gerq \mid d$
\item $\lambda \calA^{-1} \gera^{-1}\overline{\gera}$ is an integral ideal of $\ok$.
\end{enumerate}
For $\ell \in \ol$ such that $(\ell, \alpha_0 d \gera^{-1}\overline{\gera})=1$, let
$$ R = R(\gera, \lambda, \ell) = \{ [\alpha, \beta] \mid \alpha \in \calD^{-1}, \beta \in \calD^{-1}\calA^{-1} \ell \gera^{-1}\overline{\gera},
\alpha \equiv \lambda \beta \mod \ok \}. $$

\subsection{Counting simultaneous embeddings}
Let $K'$  be another CM field, with $\calO_{K'}= \ol[w]$ and $$\disc_{K'/L} = (\Tr(w)^2 -4\Norm(w)) = (d')$$
generated by a totally negative element $d'$ of $L$.
Then, following Gross-Zagier~\cite{Gross Zagier}, we are interested in counting
$$ S(\gera, \lambda, \ell) = \left\{ [\alpha,\beta] =  \begin{pmatrix} \alpha & \beta \\ \alpha_0 p \betabar & \alphabar \end{pmatrix}  \in R(\gera, \lambda, \ell) :
\Tr[\alpha, \beta] = \Tr(w), \Norm[\alpha, \beta] = \Norm(w) \right\}.
$$
We follow Gross-Zagier very closely.
Let $[\alpha, \beta]$ be an element of this set.  Since
$$\ok = \ol + \ol \cdot \frac{a+\sqrt{d}}{2} \; = \; \left\{\frac{2l_1+l_2(a+\sqrt{d})}{2}: l_1, l_2 \in \ol \right\}
$$
$$ =  \left\{\frac{l_3+l_4\sqrt{d}}{2}: l_3, l_4 \in \ol, l_3 -al_4 \equiv 0 \mod{2\ol} \right\},
$$
we can write $\alpha \in \calD^{-1}$ in the form
$\alpha = \frac{l_3+l_4\sqrt{d}}{2\sqrt{d}}$, $l_3$, $l_4 \in \ol$, with
$l_3 -al_4 \equiv 0 \mod{2\ol}$, and in this notation, $\Tr(\alpha) = \Tr([\alpha,\beta])=l_4$.
So $$\alpha = \frac{x+\Tr(w)\sqrt{d}}{2\sqrt{d}}, \quad x \in \ol, \quad x -a\Tr(w) \equiv 0 \mod{2\ol}, $$ where $a = -\Tr(t)$,
 and $$\beta = \frac{\ell}{\sqrt{d}}\gamma, \quad \gamma \in \calA^{-1} \gera^{-1}\overline{\gera}.$$
Since
\begin{equation*}
\begin{split}
\Norm[\alpha,\beta] = \det [\alpha,\beta] & = \alpha \alphabar - \alpha_0 p \beta\betabar \\
 & = \frac{x+\Tr(w)\sqrt{d}}{2\sqrt{d}} \cdot \frac{x-\Tr(w)\sqrt{d}}{-2\sqrt{d}}
 - \alpha_0 p \frac{\ell^2}{-d} \gamma \overline{\gamma} \\
& = \frac{1}{-4d}[x^2-\Tr(w)^2d-4\alpha_0 p \ell^2 \gamma \overline{\gamma}],
\end{split}
\end{equation*}
 it follows that
 $$ -d[4\Norm(w) -\Tr(w)^2] = x^2 - 4\alpha_0 p \ell^2 \gamma \overline{\gamma}. $$
So an element $[\alpha,\beta]$ of the set $S(\gera, \lambda, \ell)$ gives rise to a solution
$(x, \gamma)$ to the equation
$$ d d' = x^2 - 4\alpha_0 p \ell^2 \gamma \overline{\gamma},$$
with $ \gamma \in \calA^{-1} \gera^{-1}\overline{\gera}$, and $ x \in \ol$, $ x \equiv a\Tr(w) \mod{2\ol}$,  where $x^2-dd'$ is a totally negative element of $\ol$ because
$\alpha_0$ is. Call this set of conditions on $x$ conditions {\bf C}.

Our analysis allows us to define a function
$\phi: S(\gera, \lambda, \ell) \rightarrow S_1(\gera, x, \ell)$
that sends $[\alpha,\beta] \mapsto \gamma$ (it is used in the proof of
Theorem~\ref{SS1S2} below), where the set $S_1(\gera, x, \ell)$ is defined for an integral ideal $\gera$ and $x$ satisfying conditions {\bf C} by 
$$ S_1(\gera,x, \ell) = \{\gamma \in \calA^{-1} \gera^{-1}\overline{\gera}: \Norm(\gamma) =
\gamma \overline{\gamma} = \frac{x^2-dd'}{4\alpha_0 p \ell^2 } \}.$$

For $\gamma \in \calA^{-1} \gera^{-1}\overline{\gera}$, the ideal generated by $\gamma$ can
be written as $(\gamma) = \calA^{-1} \gera^{-1}\overline{\gera} \cdot \gerb$, for $\gerb$ an
ideal of $\ok$, and $\Norm (\gerb) = \alpha_0 \Norm(\gamma)$.
We let $S_2(\gera,x, \ell)$ be the set
$$ S_2(\gera,x, \ell) = \{ \gerb \lhd \ok:
\Norm(\gerb) = \frac{x^2-dd'}{4 p \ell^2}, \gerb \sim \gera^2\calA \}.
$$
\begin{prop}  \label{S1S2} The map from $S_1(\gera,x, \ell) \rightarrow S_2(\gera,x, \ell)$ which sends
$\gamma  \mapsto \gerb_\gamma = (\gamma)\calA \gera \overline{\gera}^{-1}$ is
a surjective $[w_K:1]$-map, where $w_K$ equals the number of roots of unity in $K$.
\end{prop}
\begin{proof}  To show that the map is $[w_K:1]$, we first show that
$\gerb_\gamma = \gerb_\delta \iff \gamma = \mu \delta$, where $\mu$ is a root of unity in $K$.
Since $\gerb_\gamma$ depends only on $(\gamma)$, the direction $\Leftarrow$ is clear.
Now if $\gerb_\gamma = \gerb_\delta$, then $(\gamma)=(\delta)$, so $\gamma = \mu\delta$
for some $\mu \in \ok^\times$, but also $\Norm(\gamma)=\Norm(\delta) = \Norm(\mu)\cdot \Norm(\gamma)
\Rightarrow \Norm(\mu) =1 \Rightarrow \mu \in \mu_K$.

Next we show that the map is surjective.  Given $\gerb \in S_2(\gera,x, \ell)$, let $\gamma$ be
a generator of $\calA^{-1} \gera^{-1}\overline{\gera}\gerb$.  Then $\gamma \in
\calA^{-1} \gera^{-1}\overline{\gera}$ and
$$(\Norm(\gamma)) = (\gamma \overline{\gamma}) = (\frac{x^2-dd'}{4\alpha_0 p \ell^2}).$$
Therefore,
there exists a totally positive unit $\epsilon' \in \ol^{\times +} = \ol^{\times 2},$
$\epsilon' = \epsilon^2,$ such that
$$\epsilon' \gamma \overline{\gamma} = \frac{x^2-dd'}{4\alpha_0 p \ell^2}.$$
Changing $\gamma$ to $\epsilon \gamma$,
$$ \gamma \overline{\gamma} = \frac{x^2-dd'}{4\alpha_0 p \ell^2}.$$
So $\gamma \in S_1(\gera,x, \ell)$, and since it is still true that
$(\gamma)=\calA^{-1} \gera^{-1}\overline{\gera}\gerb$, it follows that $\gerb_\gamma = \gerb$.
\end{proof}

Now given an element $\gamma$ of $S_1(\gera,x, \ell)$, we can construct elements of $S(\gera, \lambda, \ell)$ as follows:
Let $$\alpha = \frac{x+\Tr(w)\sqrt{d}}{2\sqrt{d}}, \quad \beta = \frac{\ell}{\sqrt{d}}\gamma.$$

First, we note that $\alpha \in \calD^{-1} \iff \frac{x+\Tr(w)\sqrt{d}}{2} \in \ok
\iff  x \in \ol, \quad x \equiv a\Tr(w) \mod{2\ol}$, which holds because $x$ satisfies conditions
{\bf C}.

Next, note that $\beta = \frac{\ell}{\sqrt{d}}\gamma \in  \calD^{-1} \calA^{-1} \ell \gera^{-1}\overline{\gera}
\iff  \gamma \in \calA^{-1} \gera^{-1}\overline{\gera}$, which holds by the definition of the set $S_1(\gera,x)$.

It remains to check that the congruence $\alpha \equiv \lambda \beta \mod{\ok}$ is satisfied.
Since $\gamma \in S_1(\gera,x, \ell)$,
$$  x^2 - 4\alpha_0 p \ell^2 \gamma \overline{\gamma} = d d' \equiv 0 \mod{d}.
$$
Next, the congruence $\lambda^2 \equiv \alpha_0 p \mod{d}$ implies that
$$  x^2 - 4\alpha_0 p \ell^2 \gamma \overline{\gamma} +
4 \ell^2 \gamma \overline{\gamma}(\alpha_0 p - \lambda^2) \equiv 0 \mod{d},
$$
and so
$$  x^2 - 4\lambda^2 \ell^2 \gamma \overline{\gamma} \equiv 0 \mod{d}.
$$
Therefore,
$$(x+\Tr(w)\sqrt{d})(x-\Tr(w)\sqrt{d}) - 4\lambda^2 \ell^2 \gamma \overline{\gamma} \equiv 0 \mod{d}.
$$
Using $x+\Tr(w)\sqrt{d} = 2\sqrt{d}\alpha$ and $\ell\gamma = \sqrt{d}\beta$, we get
$$-4d(\alpha \alphabar - \lambda^2 \beta \betabar) \equiv 0 \mod{d}.
$$
Since $(d,2) =1$ it follows that $\alpha \alphabar \equiv \lambda^2 \beta \betabar \mod{\ok}$.
Now, $\alpha$ and $\lambda \beta$ belong to $\calD^{-1} = \frac{1}{\sqrt{d}}\ok$ and hence
$$ \alpha_1 := \sqrt{d}\alpha, \quad \beta_1 := \sqrt{d} \lambda \beta$$
are in $\ok$ and we have
$ \alpha_1 \alphabar_1 \equiv \beta_1 \betabar_1 \mod{d}.$
Equivalently, this relation holds modulo all ideals $\gerq$ of $\ol$ dividing $d$:
$$ (*) \quad \quad \alpha_1 \alphabar_1 \equiv \beta_1 \betabar_1 \mod{\gerq},
\quad \forall \gerq \mid d, \quad \gerq \lhd \ol.$$
Let $\tilde{\gerq} \lhd \ok$ be a prime such that $ \gerq \ok = \tilde{\gerq}^2$.
Then $\ok/\tilde{\gerq} \cong \ol /\gerq$, and complex conjugation hence acts trivially $\mod{\tilde{\gerq}}$.
So (*) is equivalent to
$$ \alpha_1^2  \equiv \beta_1^2  \mod{\tilde{\gerq}},
\quad \forall \tilde{\gerq} \mid d\ok, \quad \tilde{\gerq} \lhd \ok,$$
which is equivalent to
$$ \alpha_1  \equiv \pm \beta_1  \mod{\tilde{\gerq}},
\quad \forall \tilde{\gerq} \mid d\ok, \quad \tilde{\gerq} \lhd \ok.$$
So this shows that there exists a choice of signs $\varepsilon(\gera, \gerq)$, and a
$\lambda$ depending on this choice, for which the congruence condition is satisfied,
and $[\alpha, \beta] \in S(\gera,\lambda, \ell)$.
However, for any ideal $\gerq$ for which $x \equiv 0 \mod{\gerq}$, both signs will work.
This motivates the following definitions and theorem.

\begin{dfn} \label{def:signs}
\begin{enumerate}
\item For $x \in \ol$, let $\delta(x) = 2^{\#\{ \gerq \mid d :x \equiv 0 \mod{\gerq}\}}$.
\item Call $\varepsilon(\gera)$ a vector of signs $\{ \varepsilon(\gera, \gerq) \}$ and let $\lambda_{\varepsilon(\gera)} \in \ol$ be an element such that
\begin{enumerate}
\item $\lambda_{\varepsilon(\gera)} \equiv \varepsilon(\gera, \gerq) \lambda_\gerq \mod \gerq$, $\forall \gerq \mid d$
\item $\lambda_{\varepsilon(\gera)} \calA^{-1} \gera^{-1}\overline{\gera}$ is an integral ideal of $\ok$.
\end{enumerate}
For any such $\lambda_{\varepsilon(\gera)}$ we have associated orders $R(\gera, \lambda_{\varepsilon(\gera)}, \ell)$.
\item Let $\tau= \#\{ \gerq \mid d \}$.
\end{enumerate}
\end{dfn}
For clarity, we also repeat previous definitions.
\begin{dfn} We say that $x \in \ol$ satisfies {\bf C} if
 $x \equiv a\Tr(w) \mod{2\ol}$, $\frac{x^2-dd'}{4 p \ell^2} \in \ol$, and $x^2-dd'$ is totally
 negative.
\end{dfn}
\begin{dfn}
For $\gera \lhd \ok$, let $\lambda_\gera = \lambda_{\varepsilon(\gera)},$
where $\varepsilon(\gera, \gerq) = (-1)^{\val_{\tilde{\gerq}}(\gera)}$.
\end{dfn}

\begin{thm} \label{SS1S2}
\begin{equation*} \tag{1}
\sum_{\varepsilon(\gera)} \# S(\gera, \lambda_{\varepsilon(\gera)}, \ell)
=
\sum_{x \, {\rm satisfies } \, {\bf C}} \delta(x) \cdot \#S_1(\gera, x, \ell)
= w_K\sum_{x \, {\rm satisfies } \, {\bf C}} \delta(x) \cdot \#S_2(\gera, x, \ell).
\end{equation*}
Furthermore,
\begin{equation*} \tag{2}
\sum_{\varepsilon(\gera)} \# S(\gera, \lambda_{\varepsilon(\gera)}, \ell)
= \sum_{\substack{\gerc \mid d \\ \gerc \lhd \ok}} \# S(\gera\gerc, \lambda_{\gera\gerc}, \ell). 
\end{equation*}

\end{thm}
\begin{proof}
To avoid confusion, we remark that in (1), the first summation is a sum over $2^\tau$ elements, one of them being $\# S(\gera, \lambda_{\gera}, \ell)$.
The second equality of (1) follows from Proposition~\ref{S1S2}.
To prove the first equality in (1), we refer to the construction given above of the map $\phi$ 
$S(\gera, \lambda, \ell) \rightarrow S_1(\gera, x, \ell)$.  It can be extended to a map from
$$ \phi: \coprod_{\varepsilon(\gera)} S(\gera, \lambda_{\varepsilon(\gera)}, \ell) \rightarrow
 \coprod_{x \, {\rm satisfies } \, {\bf C}} S_1(\gera, x, \ell).$$
 We claim that $\phi$ is a surjective map which is $[\delta(x):1]$.
Given an element $\gamma$ of $S_1(\gera,x, \ell)$, we constructed above,
for some possible choice of signs $\epsilon(\gera)$ determining $\lambda$, an element of $S(\gera, \lambda, \ell)$,
 $$\alpha = \frac{x+\Tr(w)\sqrt{d}}{2\sqrt{d}}, \quad \beta = \frac{\ell}{\sqrt{d}}\gamma.$$
For any ideal $\tilde{\gerq} \mid d$, let $\mu(x,\gamma) \in \{ \pm 1\}$ be such that
$\alpha_1 \equiv  \mu(x,\gamma) \beta_1 \mod{\tilde{\gerq}}$, where
$\alpha_1 = \sqrt{d}\alpha$, $\beta_1 = \sqrt{d} \lambda \beta$.
Given $\varepsilon(\gera)$, we have
\begin{multline*} 
\alpha \equiv \lambda_{\varepsilon(\gera)} \beta \mod{\ok} \iff \\
\forall \tilde{\gerq} \mid d,
\text{ either } \alpha_1  \equiv \beta_1 \equiv 0 \mod{\tilde{\gerq}}
\text{ or } \beta_1 \not\equiv 0 \mod{\tilde{\gerq}}
\text{ and } \varepsilon(\gera, \gerq) \equiv \mu(x,\gamma) \mod{\tilde{\gerq}}.
\end{multline*} 
%
It follows that for a given $(x,\gamma)$, the number of sign vectors $\varepsilon(\gera)$
such that $\alpha \equiv \lambda_{\varepsilon(\gera)} \beta \mod{\ok}$ is equal to
$$ 2^{\#\{ \tilde{\gerq} \mid d : \sqrt{d}\alpha  \equiv 0 \mod{\tilde{\gerq}}\}}.$$
Now since $\val_{\tilde{\gerq}}(\sqrt{d}\alpha) = \val_{\tilde{\gerq}}(x+\Tr(w)\sqrt{d})
\ge \min\{\val_{\tilde{\gerq}}(x), \val_{\tilde{\gerq}}(\Tr(w)\sqrt{d})\}$, it follows that
$$\val_{\tilde{\gerq}}(\sqrt{d}\alpha) > 0 \iff \val_{\tilde{\gerq}}(x) > 0 \iff
\val_{\tilde{\gerq}}(x) > 0.$$ So the number of sign vectors $\varepsilon(\gera)$
such that $\alpha \equiv \lambda_{\varepsilon(\gera)} \beta \mod{\ok}$ is equal to
$2^{\#\{ \gerq \mid d :x \equiv 0 \mod{\gerq}\}}$.

\noindent
The second assertion in the theorem follows from the same argument given in 
the proof of Lemma~\ref{lem:ramifiedideal}.
\end{proof}







\section{Endomorphism rings of abelian surfaces with complex multiplication} \label{sec:endring}
Let $K$ be a primitive CM field of degree $4$ over the rational numbers.
Let $W = W(\fpbar)$ be the Witt ring and let $$(A, \iota:\ok \arr \End_W(A))$$ be an abelian scheme over $W$ of relative dimension $2$, such that $A\pmod{p}$ is superspecial. Assume further that $p$ is unramified in $K$. Then, $R:= \End_{\ol}(A \pmod{p})$ is a superspecial order of the quaternion algebra $B_{p, L}$~\cite[Prop 4.1]{NicoleJNT}.

\begin{thm} \label{thm: endomorphisms mod p to the n} One has
\[ \End_{\ol, W/(p^n)}(A \pmod{p^n}) = \ok + p^{n-1}R.\]
\end{thm}
This theorem is a generalization of a theorem of B. Gross that deals with the case of elliptic curves \cite{Gross}, but our method of proof is different; it is based on crystalline deformation theory.

\

\id Consider $A\pmod{p^n}$. We have an identification:
\[ \HH^1_{dR}(A \pmod{p^n}) \cong H^1_{Crys}(A\pmod{p}/W) \otimes W/(p^n).\]
Using that $W/(p^{n+1}) \arr W/(p^n)$ has canonical divided power structure, we conclude that the deformations of $A \pmod{p^n}$ to an abelian scheme $B$ over $W/(p^{n+1})$ are in functorial correspondence with direct summands of $
H^1_{Crys}(A\pmod{p}/W) \otimes W/(p^{n+1})$ such that the following diagram commutes.
\[\xymatrix{M\ar[d]^{\mod p^{n}} & \subseteq & H^1_{Crys}(A\pmod{p}/W) \otimes W/(p^{n+1})\ar[d]^{\mod p^{n}}\\
\underline{\omega}_{A\pmod{p^n}} & \subseteq & H^1_{Crys}(A\pmod{p}/W) \otimes W/(p^{n})
},
\]
where $\underline{\omega}_{A\pmod{p^n}} $ are the relative differentials at the origin of $A\pmod{p^n}$.

We shall show that there exists a unique such $B$ to which the $\ok$-action extends, namely, a unique $M$ fixed under the $\ok$ action on $H^1_{Crys}(A\pmod{p}/W)$. We may conclude then that for that $M$ there is an isomorphism
\begin{multline}\label{equation: 512}\End_\ol(A \pmod{p^{n+1}}) \otimes_\ZZ\ZZ_p \\ \cong \End_\ol\left( M \subset H^1_{Crys}(A\pmod{p}/W)\otimes W/(p^{n+1}) \right) \cap\End_\ol \left( A \pmod{p^{n+1}} \right) \otimes_\ZZ\ZZ_p.\end{multline}
We then calculate the right hand side and find that it is equal to $(\ok + p^{n}R) \otimes_\ZZ \ZZ_p$. Since we know a priori that $\End_\ol(A \pmod{p^{n+1}})$ has index equal to a power of $p$ in $R$ (see~\cite[Proposition 6.1]{GLdenominators}), our theorem will follow.

\

\id First, the uniqueness of $M$ is easy to establish. We have an isomorphism of $\ok \otimes_\ZZ W$ modules,
\[ H^1_{crys}(A\pmod{p}/W) \cong \oplus_{\varphi \in \Emb(\ok, W)} W(\varphi), \]
where  $W(\varphi)$ is just $W$ with the $\ok$ action given by $\varphi$. Since $p$ is unramified, for all $n\geq 1$, $W(\varphi) \not\cong W(\varphi') \pmod{p^n}$ as $\ok$-modules, for any distinct $\varphi, \varphi' \in \Emb(\ok, W)$. If $\Phi$ is the CM-type of $A$ it follows that if $M$ is a direct summand of rank $g$, which is an $\ok$-submodule, then $M$ must be $\oplus_{\varphi \in\Phi} W(\varphi) \pmod{p^{n+1}}$.

Let $R_n = \End_{\ol, W/(p^n)}(A \pmod{p^n})$. We prove by induction on $n$ that
\[ R_n = \ok + p^{n-1}R.\]
As remarked, it is enough to prove that after $p$-adic completion, and, in fact, we actually calculate the right hand side of (\ref{equation: 512}). The case $n=1$ is tautological.

Since we assumed that $A\pmod p$ is superspecial and $p$ is unramified in $K$, there are according to \cite{GLdenominators}
Table 3.3.1 (ii), Table 3.4.1 (iii), (iv), Table 3.5.1 (iii), (vi), and the results of C.-F. Yu \cite{Yu}, precisely two possibilities for $H^1_{crys}(A\pmod{p}/W)$, equivalently for the Dieudonn\'e module of $A\pmod{p}$, as an $\ok\otimes_\ZZ \ZZ_p$-module. Our calculations are done separately, according to these cases.

\

\id
\fbox{Case 1}\quad In this case, the completions at $p$ of the rings are
\[ \calO_{L, p} \cong \ZZ_p \oplus \ZZ_p, \qquad  \calO_{K, p} \cong \ZZ_{p^2} \oplus \ZZ_{p^2}, \]
where we are writing $\ZZ_{p^2}$ for $W(\FF_{p^2})$. The Dieudonn\'e module $\DD$ is a direct sum of Dieudonn\'e modules,
\[ \DD = \DD_1 \oplus \DD_2, \]
where for $i=1, 2$, $\DD_i$ has a basis relative to which Frobenius is given by the matrix
\[ \begin{pmatrix} 0 & p \\ 1 & 0 \end{pmatrix},\]
and the $i$-th copy of $\ZZ_{p^2}$ in $\calO_{K, p}$ acts on $\DD_i$ by
\[a \mapsto \begin{pmatrix} a & \\ & a^\sigma \end{pmatrix}\]
and on $\DD_{i+1 \pmod{2}}$ by zero.
(Here $\sigma$ is the Frobenius automorphism of $\ZZ_{p^2}$.) Clearly,
\[ \End_\ol(\DD) = \End(\DD_1) \times \End(\DD_2), \]
and, as one can easily check,
\[ \End(\DD_i) = \left\{ \left( \begin{smallmatrix}
\alpha & p\beta \\ \beta^\sigma & \alpha^\sigma
\end{smallmatrix} \right):
\alpha, \beta \in W(\FF_{p^2})
\right\}.\]
(The restriction on the entries $\left( \begin{smallmatrix}
a & b \\ c & d
\end{smallmatrix} \right)$ comes from the identity
\[ \left( \begin{smallmatrix}
a & b \\ c & d
\end{smallmatrix}\right) \left( \begin{smallmatrix}
0 & p \\ 1& 0
\end{smallmatrix}\right) =  \left( \begin{smallmatrix}
0 & p \\ 1& 0
\end{smallmatrix}\right)
\left( \begin{smallmatrix}
a^\sigma & b^\sigma \\ c^\sigma & d^\sigma
\end{smallmatrix}\right)\]
that an endomorphism of the Dieudonn\'e module must satisfy.)

Now, for every $n$, $\underline{\omega}_{A\pmod{p^n}} = \Span_{W/(p^n)}\{ \left(\begin{smallmatrix} 0 \\ 1 \end{smallmatrix} \right) \} \oplus \Span_{W/(p^n)}\{ \left(\begin{smallmatrix} 0 \\ 1 \end{smallmatrix} \right) \}$ in the decomposition $\DD = \DD_1 \oplus \DD_2$. By induction, the endomorphisms in $\End_\ol(\DD)$ preserving $\underline{\omega}_{A\pmod{p^n}}$ are
\[ (\ok + p^{n-1}R) \otimes_\ZZ\ZZ_p = \left\{ \left(   \left(\begin{smallmatrix} \alpha & p^n\beta  \\ p^{n-1}\beta^\sigma & \alpha^\sigma \end{smallmatrix} \right),  \left(\begin{smallmatrix}  \gamma & p^n\delta  \\ p^{n-1} \delta ^\sigma & \gamma ^\sigma  \end{smallmatrix} \right)\right): \alpha, \beta, \gamma, \delta \in W(\FF_{p^2}) \right\}.\]
The conditions for an endomorphism like that to preserve $\underline{\omega}_{A\pmod{p^{n+1}}}$ are that the vectors
 $$\left(\begin{smallmatrix} \alpha & p^n\beta  \\ p^{n-1}\beta^\sigma & \alpha^\sigma \end{smallmatrix} \right)  \left(\begin{smallmatrix} 0 \\ 1 \end{smallmatrix} \right),    \left(\begin{smallmatrix}  \gamma & p^n\delta  \\ p^{n-1} \delta ^\sigma & \gamma ^\sigma  \end{smallmatrix} \right)\left(\begin{smallmatrix} 0 \\ 1 \end{smallmatrix} \right)$$ 
are a multiple of $ \left(\begin{smallmatrix} 0 \\ 1 \end{smallmatrix} \right)$ modulo $p^{n+1}$. This is the case precisely when $\beta$ (respectively, $\delta$) are in $pW$. Thus, $\End(A\pmod{p^{n+1}})\otimes_\ZZ\ZZ_p = (\ok + p^nR) \otimes _\ZZ\ZZ_p$ and the proof is complete in Case 1.

\

\id
\fbox{Case 2}\quad In this case, the completions at $p$ of the rings are
\[ \calO_{L, p} \cong \ZZ_{p^2}, \qquad  \calO_{K, p} \cong \ZZ_{p^2} \oplus \ZZ_{p^2}, \]
where $\ZZ_{p^2}$ is embedded diagonally in  $\ZZ_{p^2} \oplus \ZZ_{p^2}$. The Dieudonn\'e module has a basis $\{e_1, e_2, e_3, e_4\}$ relative to which
\[ \Fr = \begin{pmatrix}   0 & 0 & p & 0 \\
0&0&0&1\\1&0&0&0\\ 0&p&0&0
\end{pmatrix}.\]
The element $(a, b)\in \calO_{K,p}$ acts by the diagonal matrix $\diag(a, b, a^\sigma, b^\sigma)$, and so $a\in \calO_{L,p}$ acts by $\diag(a, a, a^\sigma, a^\sigma)$. Change the order of the basis elements to get a new basis $\{e_1, e_4, e_3, e_2\}$. Then Frobenius is given by
$  \begin{pmatrix}   0 & pI_2 \\ I_2 &0
\end{pmatrix},$
and $(a, b)\in \calO_{K,p}$ acts by the diagonal matrix $\diag(a, b^\sigma, a^\sigma, b)$, and so $a\in \calO_{L,p}$ acts by $\diag(a,  a^\sigma, a^\sigma, a)$.

The conditions for a matrix $ \left(\begin{smallmatrix}  A & B  \\ C& D  \end{smallmatrix} \right) \in M_4(W)$ to be in $\End(\DD)$ are  $ \left(\begin{smallmatrix}  A & B  \\ C& D  \end{smallmatrix} \right)
 \left(\begin{smallmatrix}  0 & pI_2  \\ I_2& 0  \end{smallmatrix} \right) =
  \left(\begin{smallmatrix}  0 & pI_2  \\ I_2& 0  \end{smallmatrix} \right)    \left(\begin{smallmatrix}  A^\sigma & B^\sigma  \\ C^\sigma& D^\sigma  \end{smallmatrix} \right)$ and so we find,
  \[ \End(\DD) = \left\{  \begin{pmatrix}  A & pC^\sigma  \\ C& A^\sigma  \end{pmatrix}: A, C \in M_2(W(\FF_{p^2})) \right\}.\]
  The condition for such a matrix to be in $\End_\ol(\DD)$ is that it commutes with all matrices of the form $\diag(a,  a^\sigma, a^\sigma, a)$ where $a$ runs over $W(\FF_{p^2})$. An easy computation gives
  \[ \End_\ol(\DD) = \left\{  \begin{pmatrix}  A & pC^\sigma  \\ C& A^\sigma  \end{pmatrix}: A, C \text{ diagonal matrices } \in M_2(W(\FF_{p^2})) \right\}.\]
We have $\underline{\omega}_{A\pmod{p^n}} = \Span\{e_3, e_2\}$, where $e_3, e_2$ are the last $2$ vectors in the current basis. One argues by induction, as before, to prove that the endomorphisms in $\End_\ol(\DD)$ preserving $\underline{\omega}_{A\pmod{p^n}}$ are precisely those of the form
  \[ \left\{  \begin{pmatrix}  A & p^nC^\sigma  \\ p^{n-1}C& A^\sigma  \end{pmatrix}: A, C \text{ diagonal matrices } \in M_2(W(\FF_{p^2})) \right\} \cong (\ok + p^{n-1}R) \otimes_\ZZ \ZZ_p.\]
That completes the proof of Case 2 and, hence, of the theorem.



\section{Geometric interpretation} \label{sec:geometry}

\id Let $W = W(\fpbar)$ and $Q = W \otimes_\ZZ\QQ$; $Q$ is the completion of the maximal unramified extension of $\QQ_p$. Assume that $p$ is unramified in $K$ and consider the functor on $W$-schemes associating to a $W$-scheme $S$ the isomorphism classes of triples
\begin{equation}\label{equation: triples 2} \uA = (A, \iota, \eta),
\end{equation}
where $A\arr S$ is an abelian scheme of relative dimension $g$, $\iota:\ok \arr \End_S(A)$ is a ring homomorphism and $\eta$ is a principal polarization of $A$ inducing complex conjugation on $K$. Arguments as in \cite{GL1} show that this functor is represented by an \'etale scheme over $W$, whose complex points are in natural bijection with
$\scrF \times \Cl(K)$, as described in Proposition \ref{prop: cm points}. 
In particular, isomorphism classes of $\uA$ over $\fpbar$ as in (\ref{equation: triples 2}) are also in bijection with $(\scrF \times \Cl(K))/\sim$ once we have fixed an identification of $\Hom(K, \CC)$ with $\Hom(K, \overline{\QQ_p})$.

Consider pairs $(A, \iota)$ over $\fpbar$ such that $A$ is a $g$-dimensional abelian variety and $\iota: \ok \arr \End(A)$ is a ring homomorphism such that $(A, \iota\vert_\ol)$ satisfies the Rapoport condition. One knows that there exists a principal $\ol$-polarization $\eta$ on $A$, unique up to isomorphism. We claim that $\eta$ automatically induces complex conjugation on $K$. This can be verified by case-by-case analysis using Lemma 6 of \cite{Chai}.

\subsection{Isomorphisms of CM abelian varieties}

Now, fix another CM field $K^\prime$ whose totally real subfield is $L$.  Consider $(A, \iota_A: \ok \arr \End(A))$ and $(A^\prime, \iota_{A^\prime}: \calO_{K^\prime} \arr \End(A^\prime))$ over $\fpbar$, and assume that we are given an isomorphism
\[\alpha: (A, \iota_A\vert_\ol) \overset{\sim}{\longrightarrow} (A^\prime, \iota_{A^\prime}\vert_\ol).\] We then get an embedding
\[ j_\alpha: \calO_{K^\prime} \arr \End(A), \qquad j_\alpha (r) = \alpha^{-1} \circ \iota_{A^\prime}(r)\circ \alpha.\]
If $\beta: (A, \iota_A\vert_\ol) \overset{\sim}{\longrightarrow} (A^\prime, \iota_{A^\prime}\vert_\ol)$ is another isomorphism, then $\beta = \gamma\circ \alpha$, where $\gamma \in \Aut(A^\prime, \iota_{A^\prime}\vert_\ol)$ and $j_\beta(r) = \alpha^{-1} \circ\gamma^{-1} \circ\iota_{A^\prime}(r) \circ\gamma \circ \alpha$, which gives us another embedding of $\calO_{K^\prime}$ into $\End(A)$. The embeddings are equal iff $\gamma^{-1} \circ\iota_{A^\prime}(r) \circ\gamma = \iota_{A^\prime}(r)$ for all $r\in \calO_{K^\prime}$, iff $\gamma\in \text{Cent}_{\End^0(A^\prime)}(K^\prime) \cap \Aut ((A^\prime, \iota_{A^\prime}\vert_\ol)) = \calO_{K^\prime}^\times$. (Here $\text{Cent}_{\End^0(A^\prime)}(K^\prime)$ denotes the centralizer of $K^\prime$ in $\End^0(A^\prime)$.)
Thus, each isomorphism class of $(A^\prime, \iota_{A^\prime})$ such that $(A, \iota_A\vert_\ol) \cong (A^\prime, \iota_{A^\prime}\vert_\ol)$ gives us $$\sharp (\Aut ((A^\prime, \iota_{A^\prime}\vert_\ol))/\calO_{K^\prime}^\times) =
\sharp (\Aut ((A, \iota_{A}\vert_\ol))/\calO_{K^\prime}^\times)$$ distinct embeddings of $\calO_{K^\prime}$ into $\End(A)$.

\subsection{Counting isomorphisms in the superspecial case}

Now assume we are in the superspecial reduction situation and fix an isomorphism,
\[\End_\ol(A) \cong R(\gera, \lambda_\gera),\]
for some unique $\gera \lhd \ok$ (Lemma~\ref{lem:ordersRa}, Theorem~\ref{classification}).  Then, writing $\calO_{K^\prime} = \ol[\omega]$ as before, to give an embedding
$\calO_{K^\prime} \longrightarrow \End_\ol(A)$ is to choose an element $[\alpha, \beta] \in R(\gera, \lambda_\gera)$
with trace equal to $\Tr(\omega)$ and with norm equal to $\Norm(\omega)$.  That is, an element of the set
$S(\gera, \lambda_\gera,1)$.
Such an embedding makes $(A, \iota_A\vert_\ol)$ into an abelian variety with CM by $\calO_{K^\prime}$,
and so the embedding $\calO_{K^\prime} \longrightarrow \End_\ol(A)$ arises via a particular isomorphism
\[(A, \iota_A: \ok \rightarrow \End(A)) \overset{\sim}{\rightarrow} (A^\prime, \iota^\prime: \calO_{K^\prime} \longrightarrow \End(A^\prime))\]
(where, in fact, we may take $A=A^\prime$ and $\iota^\prime$ restricts to $\iota_A$ on $\ol$).  We conclude that
\[\frac{\sharp S(\gera, \lambda_\gera,1)}{\sharp (R(\gera, \lambda_\gera)^\times/\calO_{K^\prime}^\times)} =
\sharp \{ (A^\prime, \iota_{A^\prime}: \calO_{K^\prime} \rightarrow \End_\ol(A^\prime))/\fpbar: (A^\prime, \iota_{A^\prime}\vert_\ol ) \overset{\sim}{\rightarrow} (A, \iota_A\vert_\ol)
\}.
\]
(where on the left hand side, we consider $(A^\prime, \iota_{A^\prime}: \calO_{K^\prime} \rightarrow \End_\ol(A^\prime))$ up to isomorphism with CM by  $\calO_{K^\prime}$, of course).
Exactly the same analysis is valid over $W/(p^n)$, and using that
$\End_{W/(p^n)}(A,\iota \vert \ol) \cong R(\gera, \lambda_\gera, p^{n-1})$, as follows from Theorem~\ref{thm: endomorphisms mod p to the n},
we get that
\begin{equation} \frac{ \sharp S(\gera, \lambda_\gera, p^{n-1})}{\sharp (R(\gera, \lambda_\gera, p^{n-1})^\times/\calO_{K^\prime}^\times)} =
\sharp \{ (A^\prime, \iota_{A^\prime}: \calO_{K^\prime} \rightarrow \End_\ol(A^\prime))/\fpbar:
(A^\prime, \iota_{A^\prime}\vert_\ol ) \overset{\sim}{\rightarrow} (A, \iota_A\vert_\ol)
\}.
\end{equation}

\subsection{Counting formulas for the number of isomorphisms for superspecial CM types}

Now fix a superspecial CM type $\Phi$ of $K$.  We consider representatives $\underline{A} =  (A, \iota_A: \ok \rightarrow \End(A))$ for the isomorphism classes with CM type $\Phi$.  For each such $\underline{A}$,
we may choose an isomorphism
$$ f_{\underline{A}} : \End_L^0(\underline{A}) \overset{\sim}{\longrightarrow} B_{p,L},$$
and hence get an embedding
$$f_{\underline{A}}\circ \iota_A : K \rightarrow B_{p,L}.$$
By Skolem-Noether, we may conjugate the identifications $f_{\underline{A}}$ so that the embeddings $ f_{\underline{A}}\circ \iota_A$
are {\emph the same}, and in fact, this will be the case if $f_{\underline{A}_1}$ and $f_{\underline{A}_2}$ are
related by a CM isogeny to begin with.
Then, for every $\underline{A}$, $f_{\underline{A}}(\End_\ol(\underline{A}))$ is a superspecial order containing $\ok$.
This order is uniquely determined by $\underline{A}$, up to conjugation by $K^\times$.

By our results, the representatives for these orders modulo conjugation by $K^\times$ are precisely the orders
$R(\gera,\lambda_\gera)$ as $\gera$ ranges over representatives for $\Cl(\ok)$.  We therefore conclude:
\begin{thm}
\begin{multline}\label{fixedCM}
\sum_{\gera} \# S(\gera, \lambda_{\gera}, p^{n-1})
= \\ \sum_{\tiny{\begin{tabular}{ll}$\underline{A}/(W/(p^n))$ \\with CM type $\Phi$\end{tabular}}} \# \left(\frac{\End_{\ol, W/(p^n)}(\underline{A})^\times}{\calO_{K^\prime}^\times} \right)\cdot
\#\left\{\begin{tabular}{ll} $\underline{A}^\prime$ with CM by $\calO_{K^\prime}$ such that\\
$(A^\prime, \iota_{A^\prime}\vert_\ol ) \cong (A, \iota_A\vert_\ol)$\end{tabular}
\right\}.
\end{multline}
(Of course, the $\underline{A}^\prime$ are taken up to isomorphism.)
\end{thm}
If we wish not to fix a CM type on $K$, we get the following:
\begin{thm} \label{thm:geometry}
\begin{multline}\label{anyCM}
(\# \text{\rm superspecial CM types} ) \times \sum_{\gera} \# S(\gera, \lambda_{\gera}, p^{n-1})
= \\ \sum_{\tiny{\begin{tabular}{ll}$\underline{A}/(W/(p^n))$ \\with CM by $\ok$\end{tabular}}} \# \left(\frac{\End_{\ol, W/(p^n)}(\underline{A})^\times}{\calO_{K^\prime}^\times} \right)\cdot
\#\left\{\begin{tabular}{ll} $\underline{A}^\prime$ with CM by $\calO_{K^\prime}$ such that\\
$(A^\prime, \iota_{A^\prime}\vert_\ol ) \cong (A, \iota_A\vert_\ol)$\end{tabular}
\right\}.
\end{multline}
\end{thm}

\subsection{Counting formulas for pairs of embeddings into superspecial orders} \label{sec: triples}

The left hand side of (\ref{fixedCM}), for $n=1$, has another interpretation.  Consider a pair of embeddings
$\iota:\ok \rightarrow R$ and $\iota^\prime:\calO_{K^\prime} \rightarrow R$ into a superspecial order $R$
such that both restrict to a fixed, given embedding of $\ol$ into $R$.
We call it an optimal triple $(\iota, \iota^\prime, R)$.  We say that $(\iota, \iota^\prime, R)$ are conjugate
to $(j, j^\prime, \tilde{R})$ if there exists $t \in B_{p,L}^\times$ such that $t^{-1}Rt = \tilde{R}$ and
$t^{-1} \iota(x)  t = j(x)$, for all $x \in \ok^\times$ and
$t^{-1} \iota^\prime(x)  t = j^\prime(x)$, for all $x \in \calO_{K^\prime}^\times$.

To count the number of conjugacy classes of triples, let us fix an embedding $I:K \rightarrow B_{p,L}$.
Then any optimal triple is conjugate to $(I \vert \calO_{K^\prime}, \iota^\prime, R)$, where R is a superspecial order containing $I(\ok)$.   We may still conjugate by $K^\times$  and so assume that
$R=R(\gera, \lambda_\gera)$ for some $\gera$.  We may still conjugate by $\ok^\times$ and if
$K \ne K^\prime$ that induces a faithful action of $\ok^\times/\ol^\times$ on the embeddings
$\iota^\prime : \calO_{K^\prime} \rightarrow R(\gera, \lambda_\gera)$ if they exist at all.
We conclude that
\[\# (\ok^\times/\ol^\times)^{-1}  \sum_{\gera} \# S(\gera, \lambda_{\gera}, 1) = \# \{ \text{\rm optimal triples up to conjugation}
\}.\]
Finally, we note the following corollary:
\begin{cor}
\begin{multline}\label{123}
\# \{ \text{\rm optimal triples up to conjugation}\} = \\
\# (\ok^\times/\ol^\times)^{-1}  \sum_{\gera} \# S(\gera, \lambda_{\gera}, 1)  = \\
 \sum_{\tiny{\begin{tabular}{ll}$\underline{A}/(W/(p^n))$ \\with CM type $\Phi$\end{tabular}}}
 \# (\ok^\times/\ol^\times)^{-1}\# (\calO_{K^\prime}^\times/\ol^\times)^{-1}
  \# \left(\frac{\End_{\ol, W/(p^n)}(\underline{A})^\times}{\calO_{L}^\times} \right)\times \\
\#\left\{\begin{tabular}{ll} $\underline{A}^\prime$ with CM by $\calO_{K^\prime}$ such that\\
$(A^\prime, \iota_{A^\prime}\vert_\ol ) \cong (A, \iota_A\vert_\ol)$\end{tabular}
\right\}.
\end{multline}
\end{cor}
If we multiply the whole set of equalities (\ref{123}) above by the number of superspecial types for $K$, we may be justified in calling the new right hand side of (\ref{123}) the ``coincidence number of $K$ and $K'$ at $p$", as it counts the number of coincidences between abelian varieties with CM by $K$ and abelian varieties with CM by $K'$ in characteristic $p$, once one considers them as abelian varieties with RM only.

\section{The connection to moduli spaces} \label{sec: moduli}

In their paper \cite{Gross Zagier}, Gross and Zagier give a beautiful formula. Let $E_1$ and $E_2$ be two elliptic curves over $W = W(\fpbar)$.
Let $j_i$ be the $j$-invariant of $E_i$. Their formula is:
\[ \val_p(j_1 - j_2) = \frac{1}{2} \sum_{n\geq 1} \sharp \text{Isom}_n(E_1,E_2),  \]
where $\text{Isom}_n$ denotes the isomorphisms between the reduction of $E_i$ modulo $(p^n)$.

The proof Gross and Zagier provided is through direct manipulations of Weierstrass equations. A more conceptual proof was given by Brian Conrad in \cite{Conrad}. The proof makes essential use of moduli spaces, but uses many features unique to modular curves and, hence, is not readily amenable to generalization. This result is the basis of interpreting their theorem on $J(d, d^\prime)$ and $\ord_\lambda(J(d, d^\prime))$ (cf. Introduction), as an arithmetic intersection number. It thus remains a question of how to give an interpretation for our theorems, Theorem~\ref{thm:geometry} for example, as an intersection number of CM points on Shimura varieties. 

One possibility is to use Shimura curves associated with quaternion algebras over totally real fields, split at exactly one infinite prime. This approach entails using the $p$-adic, not-quite-canonical, models for these Shimura curves, following Morita, Carayol and Boutot-Carayol. The other possibility is to view these CM $0$-cycles as lying on a Hilbert modular variety. This approach is complicated by the fact that there is no ``robust" definition of the arithmetic intersection of $0$-cycles ($1$-cycles on the arithmetic models) once their co-dimension is bigger than $1$. This calls for an ad-hoc approach and it has its own challenging problems. 

For now we will replace the notion of an intersection number with something less precise, and define instead a {\it coincidence number}, which does not reflect the power to which various primes may appear in the differences of invariants, but at least reflects whether a prime appears or not in the factorizations of the differences of invariants.  In Section~\ref{sec:example} we will give an example to illustrate the coincidence number in computations.

Let $L$ be a totally real field with strict class number $1$, and $K_i, i=1, 2,$ two CM fields containing $L$ as their maximal totally real subfield. Let $p$ be a prime, unramified in both $K_1, K_2$. For each CM field we can associate a zero cycle, $\text{CM}(K_i)$, on the generic fiber of the Hilbert modular variety $\calH_L$ parameterizing principally polarized abelian varieties with RM by $\ol$ (see Section~\ref{sec:CMpoints}). Each point $x_\eta$ in $\text{CM}(K_i)$ can be extended to a $W(\fpbar)$-point $x$ on $\calH_L$ (see \cite[Lemma 2.3]{GLdenominators}). This implicitly depends on a choice of a prime $\gerp$ in a common field of definition for all the CM abelian varieties under consideration. We write $\text{CM}(K_1)  = \sum_i x_i$, $\text{CM}(K_2)  = \sum_j y_j$. We then define the arithmetic \emph{coincidence number} (for lack of better terminology) of $\text{CM}(K_1) $ and $\text{CM}(K_2)$ as
\[ \text{CM}(K_1) _\wedge \text{CM}(K_2) = \sum_{ij} x_i {_\wedge} y_j\]
where $x_i {_\wedge} y_j$ is defined as 
$1$ if $x_i$ and $y_j$  have isomorphic reduction modulo $p$, and as zero otherwise. In this notation, Theorem~\ref{thm:geometry} implies the following: 
\begin{cor} The contribution from a prime $p$ of superspecial reduction to $\text{CM}(K_1) _\wedge \text{CM}(K_2)$ is equal to $(\# \text{\rm superspecial CM types} ) \times \sum_{\gera} \# S(\gera, \lambda_{\gera}, 1)$.\footnote{Likewise, the notion of superspecial CM types depends on the implicit choice of $\gerp$.} This number, and in particular whether it is zero or not, can be effectively calculated.
\end{cor}


\section{Supersingular orders} \label{supersingularorders}

\begin{thm} Let $p$ be a rational prime and $k$ an algebraically closed field of characteristic $p$. Let $K$ be a quartic CM field and let $L = K^+$ be its real subfield.
 Let $A/k$ be an abelian surface which is supersingular, but not superspecial, with complex multiplication by $\ok$. Let $\calO = \End_{\ol}(A)$, where the endomorphisms are over $k$. Let $B_{p, \infty}$ be the quaternion algebra over $\QQ$ ramified at only $p$ and $\infty$ and let $B_{p, L} = B_{p, \infty} \otimes_\QQ L$. Then $\calO$ is an Eichler order of $B_{p, L}$ of discriminant $p^2$.
 \end{thm}
 \begin{proof} 
Let $H$ be a quaternion algebra over a number field $F$ and let $R$ be an order of $H$, containing $\calO_F$. Recall that $R$ is called an Eichler order if it is the intersection of two maximal orders. This is a local property~\cite[p. 84]{Vigneras}. If $F$ denotes now a non-archimedean local field with uniformizer $\pi$, then an order of $H$, containing $\calO_F$, is Eichler (namely, is the intersection of two maximal orders of $H$) if and only if it is conjugate to the order 
\[ M = \begin{pmatrix} \calO_F & \calO_F \\ \pi^n \calO_F & \calO_F   \end{pmatrix},\]
for some positive integer $n$~\cite[p. 39]{Vigneras}.

We wish to find the completion of $\calO$ at every rational prime ideal $\gerl$ of $\ol$.

First, since there exists an isogeny of degree a power of $p$ between any two supersingular abelian surfaces $A, A^\prime$, with real multiplication, respecting the real multiplication structure \cite{BG}, for $\gerl \nmid p$, we have $\calO_{\gerl}:= \calO\otimes_{\ol} \calO_{L, \gerl} \cong \calO^\prime_{\gerl}$, where $\calO^\prime =  \End_{\ol}(A^\prime)$. We may choose for $A^\prime$ the surface $E \otimes_\ZZ\ol$, where $E$ is a supersingular elliptic curve with $R = \End(E)$ a maximal order in $B_{p, \infty}$. Then $ \calO^\prime = \End(A^\prime) = R \otimes_\ZZ \ol$ and so $ \calO^\prime$ and $\calO$ are maximal orders at $\gerl$.

 We remark that according to the classification of the reduction of abelian surfaces with CM, the situation we consider occurs if and only if $p$ is inert in $K$. That is, in the following cases:
 \renewcommand{\labelenumi}{(\alph{enumi})}
 \begin{enumerate}
 \item
 $K/\QQ$ is cyclic Galois and $p$ is inert in $K$ (case (iii) in Table 3 of \cite{GLdenominators});
 \item
 $K/\QQ$ is non-Galois and $p$ is inert in $K$ (case (vii) in Table 5  of \cite{GLdenominators}).
 \end{enumerate}
 \id Following the conventions of \cite{GLdenominators}, the Dieudonn\'e module of the $p$-divisible group of the reduction of $A$ modulo $\gerp_L$ is
 \[ \DD \cong \WW(1)\oplus \WW(y^2)\oplus \WW(y) \oplus \WW(y^3),\]
 where $\WW(\alpha)$ denotes the Witt vectors of $\fpbar$ where $\ok$ acts through the embedding
$\alpha: K \arr \Qpbar$. Let $\sigma$ denote the Frobenius automorphism of $\WW$. Then:
  \renewcommand{\labelenumi}{(\alph{enumi})}
 \begin{enumerate}
 \item
$\ol$ acts on $\DD$ by $\ell \mapsto \diag(\ell, \ell, \sigma(\ell), \sigma(\ell))$, and
\item
$\ok$  acts on $\DD$ by $k \mapsto \diag(k, \sigma^2(k), \sigma(k), \sigma^3(k))$.
  \end{enumerate}
 \id The $p$-adic CM type is $\{1, y^3\}$, according to our conventions, but since the situation is symmetric, we may assume that the $p$-adic CM type is  $\{1, y\}$, and so Frobenius is given in the standard basis by the matrix
 \[ \Fr = \begin{pmatrix} 0 & 0 & 0 & 1 \\ 0 & 0 & p& 0 \\ p & 0 & 0 & 0 \\ 0 & 1 & 0 & 0\end{pmatrix}.\]
By a theorem of Tate, $\End(A)\otimes_\ZZ \ZZ_p \cong \End(\DD)$, where on the right the endomorphisms are as Dieudonn\'e modules (cf. \cite[Theorem 5]{WM}): namely, in this case, $\WW$-linear maps $\DD \arr \DD$ that commute with Frobenius. In the same way, \[\calO_p = \End_{\ol}(A)\otimes_{\ol} \ol_p = \End_{\ol}(A)\otimes_{\ZZ} \ZZ_p \cong  \End_{\ol}(\DD).\]

Since $\calO_p$ commutes with $\ol$, one finds that $\calO_p$ is given by block diagonal matrix with blocks of size $2$. Writing the general such matrix as
\[ M = \begin{pmatrix} m_{11} & m_{12} & & \\ m_{21} & m_{22} & & \\ & & n_{11} & n_{12} \\ & & n_{21} & n_{22} \end{pmatrix},\]
 the condition $M \cdot \Fr = \Fr \cdot \sigma(M)$ gives, after a short computation,
 \[\calO_p = \left\{   \begin{pmatrix} m_{11} & m_{12} & & \\ p^2m_{12}^{\sigma^2} & m_{11}^{\sigma^2} & & \\ & & m_{11}^{\sigma} & pm_{12}^{\sigma} \\ & & pm_{12} ^{\sigma^3}& m_{11}^{\sigma^3} \end{pmatrix}: m_{ij} \in \WW(\FF_{p^4}) \right\}.\]
Since $p$ is inert in $L$, the quaternion algebra $B_{p, L}$ is ramified only at the two places at infinity. In particular, $B_{p, L} \otimes_L L_p \cong M_2(\QQ_{p^2})$, where $\QQ_{p^2} = \WW(\FF_{p^2}) \otimes_\ZZ \QQ$. To determine the nature of $\calO_p$, we want to recognize it as a suborder of $M_2(\WW(\FF_{p^2}))$.

 \

  \id {\bf The case $p \neq 2$.\quad } Put
  \[ i = \begin{pmatrix} & 1 \\ p^2 & \end{pmatrix}, \qquad j = \begin{pmatrix} \alpha & \\ & \alpha^{\sigma^2}\end{pmatrix}, \]
  where $\alpha$ is chosen such that $\WW(\FF_{p^4}) = \WW(\FF_{p^2})[\alpha]$ and  $\alpha^{\sigma^2} = -\alpha$.
 We have then
  \[ i^2 = p^2,\qquad  j^2 = \alpha^2, \qquad k:= ij = -ji = \begin{pmatrix} & -\alpha \\ p^2 \alpha \end{pmatrix}.\]
  Writing $m_1 = x_1 + y_1\alpha, m_2 = x_2 + y_2\alpha$ with $x_i, y_i \in \WW(\FF_{p^2})$ we can write,
 \begin{equation*}
  \begin{split} \begin{pmatrix} m_{11} & m_{12} \\ p^2m_{12}^{\sigma^2} & m_{11}^{\sigma^2} \end{pmatrix} & =
  x_1  \begin{pmatrix} 1&  \\  & 1\end{pmatrix} + y_1  \begin{pmatrix} \alpha &  \\  & \alpha^{\sigma^2}\end{pmatrix}  + x_2  \begin{pmatrix} & 1 \\ p^2 & \end{pmatrix} - y_2  \begin{pmatrix} & -\alpha \\ p^2 \alpha \end{pmatrix}\\
  & = x_1 \cdot 1 + y_1\cdot j + x_2\cdot i - y_2 \cdot k.
  \end{split}
 \end{equation*}
 Conversely, for any $x_i, y_i\in \WW(\FF_{p^2})$ we get an element of $\calO_p$. Thus,
 \[ \calO_p = \WW(\FF_{p^2})\cdot 1 \oplus \WW(\FF_{p^2}) \cdot i \oplus \WW(\FF_{p^2}) \cdot j \oplus \WW(\FF_{p^2}) \cdot k.\]

  Let $I = p^{-1} i, J = j, K = IJ = -JI$. Then $I^2 = 1, J^2 = \alpha^2, K^2 = -\alpha^2$. The module
  \[ R = \WW(\FF_{p^2})[1, I, J, K]\]
  is in fact an order of $M_2(\QQ_{p^2})$ and it has discriminant $1$. It must then be isomorphic to $M_2(\WW_{p^2})$, and, indeed, if we send
  \[ 1 \mapsto  \begin{pmatrix} 1&  \\  & 1\end{pmatrix}, \quad I \mapsto  \begin{pmatrix} 1& \\  & -1 \end{pmatrix}, \quad J
 \mapsto  \begin{pmatrix} & \alpha^2 \\ 1 & \end{pmatrix}, \quad  K \mapsto \begin{pmatrix} & \alpha^2 \\ -1 & \end{pmatrix}\]
 we get the isomorphism $R \cong M_2(\WW(\FF_{p^2}))$. Under this isomorphism $\calO_p$ is mapped isomorphically to the order spanned over $\WW(\FF_{p^2})$ by the matrices $\begin{pmatrix} 1&  \\  & 1\end{pmatrix}, \begin{pmatrix} p& \\  & -p \end{pmatrix}, \begin{pmatrix} & \alpha^2 \\ 1 & \end{pmatrix}, \begin{pmatrix} & p\alpha^2 \\ -p & \end{pmatrix}$, which can be described as
 \[ \left\{ \begin{pmatrix} a & b \\ c & d\end{pmatrix} : a, b, c, d \in \WW(\FF_{p^2}), p\vert(a-d), p\vert(b - \alpha^2c) \right\}.\]
 Now conjugate $\calO_p$ by the matrix $A =  \begin{pmatrix} 1 & \alpha \\ \alpha^{-1} & -1\end{pmatrix}$. Using
 \[ 2 A^{-1}  \begin{pmatrix} a & b \\ c & d\end{pmatrix}  A =
 \begin{pmatrix} a + \alpha^{-1} b + \alpha c + d  & \alpha(a-d) + (\alpha^2c - b) \\
\alpha^{-1}(a - d) + \alpha^{-2}(b - \alpha^2c) & a - \alpha^{-1}b - \alpha c + d
 \end{pmatrix} ,\]
 we find that $\calO_p$ is conjugate to a suborder of
 \[  R' = \begin{pmatrix} \WW(\FF_{p^2}) & p \WW(\FF_{p^2})\\ p\WW(\FF_{p^2}) & \WW(\FF_{p^2})\end{pmatrix}.\]
 However, comparing the discriminant of $\calO_p$, which is $p^2$, and of $R'$ which is $p^2$ as well, we conclude that $\calO_p$ is isomorphic to $R'$. Further conjugation by the matrix $ \begin{pmatrix}  & 1/p\\ 1 &\end{pmatrix}$ shows that $\calO_p$ is isomorphic to the order
 \[ R'' = \left\{  \begin{pmatrix} a & b\\ c &d\end{pmatrix}: a, b, c, d \in \WW(\FF_{p^2}), p^2 \vert c\right\}, \]
 which is an Eichler order of discriminant $p^2$.

   \

  \id {\bf The case $p = 2$.\quad}  We may find $\alpha\in \WW(\FF_{p^2})$ such that $\alpha^{\sigma^2} = -\alpha$ and $\WW(\FF_{p^4}) =  \WW(\FF_{p^2})[\frac{1+\alpha}{2}]$. Indeed, for a suitable $\epsilon\in \WW(\FF_{p^2})^\times$ we have $\WW(\FF_{p^4}) =  \WW(\FF_{p^2})[\beta]$, where $\beta^2+\beta + \epsilon = 0$. Note that $\beta$ is a unit. Take $\alpha = -(2\beta+ 1)$.

 To make the analogy with the previous case more visible, we keep using $p$ instead of $2$ in most places.  As before, we let
  \[ i = \begin{pmatrix} & 1 \\ p^2 & \end{pmatrix}, \qquad j = \begin{pmatrix} \alpha & \\ & -\alpha\end{pmatrix}, \qquad k = ij = -ji = \begin{pmatrix} & -\alpha  \\  \alpha p^2 & \end{pmatrix}. \]
  Writing $m_1 = x_1 + y_1(1+\alpha)/2, m_2 = x_2 + y_2(1+\alpha)/2$ with $x_i, y_i \in \WW(\FF_{p^2})$ we can write,
\[ \begin{pmatrix} m_{11} & m_{12} \\ p^2m_{12}^{\sigma^2} & m_{11}^{\sigma^2} \end{pmatrix} = x_1\cdot 1 + y_1\cdot\frac{1+j}{2} + x_2\cdot i + y_2\cdot\frac{i-k}{2},\]
and one concludes that
\[ \calO_p = \WW(\FF_{p^2})\cdot 1 \oplus  \WW(\FF_{p^2})\cdot i  \oplus  \WW(\FF_{p^2})\cdot \frac{1+j}{2}  \oplus  \WW(\FF_{p^2})\cdot \frac{i-k}{2}.\]
One can verify directly that the right hand side is indeed an order and its discriminant is $p^2$.

The order $\calO_p$ contains the order 
$\WW(\FF_{p^2}) [1, i, j, k] = \WW(\FF_{p^2}) [1, I, J, K]$, 
where $I = i, J = j/\alpha, K = k/\alpha$. 
Note that $I^2 = p^2, J^2 = 1, K^2 = -p^2, IJ = -JI = K$. 
Consider the linear map 
$$\WW(\FF_{p^2}) [1, I, J, K] \arr M_2(\WW(\FF_{p^2}))$$ 
determined by \[1 \mapsto \begin{pmatrix} 1 & \\ & 1\end{pmatrix}, \quad I  \mapsto \begin{pmatrix}  & 2\\ 2& \end{pmatrix}, \quad J  \mapsto \begin{pmatrix} 1 & \\ & -1\end{pmatrix}, \quad K  \mapsto \begin{pmatrix} & -2\\ 2& \end{pmatrix}.\] One checks that this map is a ring homomorphism and verifies that
\[ \calO_p \cong \WW(\FF_{p^2}) \left[\begin{pmatrix} 1 & \\ & 1\end{pmatrix}, \begin{pmatrix} (1+\alpha)/2 & \\ & (1-\alpha)/2\end{pmatrix}, 2\begin{pmatrix}  & 1\\ 1& \end{pmatrix}, 2\begin{pmatrix}  &(1+\alpha)/2  \\  (1-\alpha)/2& \end{pmatrix}\right].\]
Let $u = (1+\alpha)/(1 - \alpha) = \beta^2/\epsilon$. Then $u$ is a unit and $1-u = 2 + u/\beta$ is a unit as well. It follows that,
\[ \calO_p \cong \WW(\FF_{p^2}) \left[\begin{pmatrix} 1 & 0\\ 0&0 \end{pmatrix}, \begin{pmatrix} 0 &0 \\0 & 1\end{pmatrix}, 2\begin{pmatrix}  0& 1\\ 0&0 \end{pmatrix}, 2\begin{pmatrix}  0&0 \\  1&0 \end{pmatrix}\right] = \left\{ \begin{pmatrix} a & b \\ c & d\end{pmatrix} :a, b, c, d\in \WW(\FF_{p^2}), p\vert b, p \vert c\right\}.\]
An additional conjugation as in the case $p\neq 2$ shows that this is an Eichler order of discriminant $p^2$.
\end{proof}

\section{A crude version of Gross-Zagier's result on singular moduli} \label{sec:bound}

\

\id Let $A$ be a $g$-dimensional abelian variety over a field $k$. Let $L$ be a totally real field of degree $g$ over $\QQ$ of strict class number one, and let $K_i$, $i=1, 2$, be two CM fields contained in some algebraic closure of $L$ such that $K_1^+ = K_2^+ = L$. We allow $K_1 = K_2$. Assume we are given two embeddings,
\[ \varphi_i: K_i \arr \End^0_k(A):= \End_k(A) \otimes_\ZZ \QQ, \]
such that
\[ \varphi_1\vert_L = \varphi_2\vert_L, \qquad \varphi_1(K) \neq \varphi_2(K).\]
\begin{lem} \label{mustbessing} The field $k$ has positive characteristic $p$. The abelian variety is supersingular and $\End^0(A) \cong B_{p, L}$, where
$B_{p, L}  = B_{p, \infty} \otimes_\QQ L$ and $B_{p, \infty}$ is ``the" quaternion algebra over $\QQ$ ramified at $p$ and $\infty$.
\end{lem}
\begin{proof} This follows easily from the classification of the endomorphism algebras of abelian varieties with real multiplication as in \cite[Lemma 6]{Chai}; one observes that under our assumptions the centralizer of $L$ in $\End^0_k(A)$ is an $L$-vector space of dimension greater than $2$.
\end{proof}

\id Let $\calO_i \subseteq K_i$ be orders containing $\ol$. The order $\calO_i$ is determined by its conductor $\gerc_i$, which is an integral ideal of $\ol$ for which we choose a generator $c_i$ (see \cite[Lemma 4.1]{GorenLauterDistance}). In fact, one can write \[\calO_{K_i} = \ol[\kappa_i],\] where $\kappa_i$ satisfies a quadratic equation $x^2 + B_ix + C_i$, $B_i, C_i \in \ol$, and $-m_i = B_i^2 - 4 C_i$ is a totally negative element of $\ol$. The relative different ideal $\calD_{K_i/L}$ is equal to $\calO_{K_i}[1/\sqrt{-m_i}]$ (\cite[Lemma 3.1]{GL2}). We have $\calO_{K_i} = \ol[\kappa_i] \supseteq \ol[\sqrt{-m_i}] \supseteq \ol[2\kappa_i]$, and so \[\calO_i = \ol[c_i\kappa_i] \supseteq \ol[c_i\sqrt{-m_i}] \supseteq \ol[2c_i\kappa_i].\] The discriminant of $\calO_i$ relative to $\ol$, $\disc_{K_i/L}(\calO_i)$, is equal to the $\ol$-ideal generated by $c_i^2m_i$ and the discriminant of $\calO_i$ relative to $\ZZ$, $\disc(\calO_i) = \disc_{K/\QQ}(\calO_i)$, is equal to $\Norm_{L/\QQ}(c_i^2m_i)\cdot \disc(\ol)^2$. (In general, we use ``$\disc$" to denote absolute discriminant, that is, relative to $\ZZ$.)

\

Let $B$ be any totally definite quaternion algebra over $L$, that is $B\otimes_{L, \sigma}\RR$ is a division algebra for any embedding $\sigma\colon L \arr \RR$,  and let $\gerd$ be its discriminant. Let
\[\varphi_i: K_i \arr B,\]
be two embeddings such that $\varphi_1\vert_L = \varphi_2\vert_L$ and $\varphi_1(K_1) \neq \varphi_2(K_2)$. Let
\[ k_i = \varphi_i(c_i \sqrt{-m_i}).\]
Let $\calO$ be an order of $B$, which we assume to contain $\varphi_i(\calO_i)$, $i= 1, 2$, and hence also $\ol$ (we view $\varphi_i$ as the identity maps on $L$). Let $\gerd^+$ be the discriminant of $\calO$. As in \cite{GL1}, subject to the assumption $\varphi_1(K_1) \neq \varphi_2(K_2)$, one proves the following lemma.

\begin{lem} The $\ol$ module $\Lambda = \ol + \ol k_1 + \ol k_2 + \ol k_1k_2 $ has finite index in $\calO$ and is in fact a direct sum, $\Lambda = \ol \oplus \ol k_1 \oplus \ol k_2 \oplus \ol k_1k_2 $ .
\end{lem}

\begin{thm} \label{crudebound} Let $\alpha = {\rm Trd}(k_1k_2)$. Then we have a divisibility of integral ideals in $L$:
\[ \gerd^+ \vert  \left( 4 {\rm Nrd}(k_1){\rm Nrd}(k_2) - \alpha^2 \right)  \qquad (\text{in } \ol).\]
Furthermore,
\[ N_{L/\QQ} (\gerd^+) \leq 4^g \frac{{\disc(\calO_1)}\cdot{\disc(\calO_2)}}{\disc(\ol)^4}.\]
\end{thm}
\begin{proof}
The discriminant of the order $\Lambda$ relative to $L$, $\disc_{B/L}(\Lambda)$,  is divisible by the discriminant of $\calO$, namely it is an integral ideal of $L$ divisible by $\gerd^+$. Using the basis $1, k_1, k_2, k_1k_2$ for $\Lambda$, and putting $\alpha = {\rm Trd}(k_1k_2)$, we find that the discriminant of $\Lambda$ is the $\ol$-ideal generated by
\[ \det \begin{pmatrix} 2 & 0 & 0 & \alpha \\ 0 & 2{\rm Nrd}(k_1) & - \alpha & 0 \\ 0 & -\alpha & 2{\rm Nrd}(k_2) & 0
\\ \alpha & 0 & 0 & 2{\rm Nrd}(k_1){\rm Nrd}(k_2)
\end{pmatrix} = \left( 4 {\rm Nrd}(k_1){\rm Nrd}(k_2) - \alpha^2 \right)^2,\]
and so
\[ \gerd^+ \vert  \left( 4 {\rm Nrd}(k_1){\rm Nrd}(k_2) - \alpha^2 \right)  \qquad (\text{in } \ol).\]
Thus,
\[  N_{L/\QQ} (\gerd^+) \vert N_{L/\QQ}  \left( 4 {\rm Nrd}(k_1){\rm Nrd}(k_2) - \alpha^2 \right)   \qquad (\text{in } \ZZ).\]
Now, $4 {\rm Nrd}(k_1){\rm Nrd}(k_2) - \alpha^2$ is a totally positive element of $\ol$. Indeed,
this is just the Cauchy-Schwartz inequality applied to the bilinear form $\Trd(x \bar y)$ under
every embedding $L \arr \RR$. We can therefore conclude that
\[  N_{L/\QQ} (\gerd^+) \leq  N_{L/\QQ}  \left( 4 {\rm Nrd}(k_1){\rm Nrd}(k_2) \right).\]
We conclude that
\begin{equation*}
\begin{split}
N_{L/\QQ} (\gerd^+) & \leq \disc(\ol)^{-4} 4^{-g} \prod_{i=1}^2 4^g \disc(\ol)^2 N_{L/\QQ} {\rm Nrd}(k_i) \\
& \leq \disc(\ol)^{-4} 4^{-g} \prod_{i=1}^2  \disc(\ol[2c_i\kappa_i])\\
& =  \disc(\ol)^{-4} 4^{g} \prod_{i=1}^2  \disc(\ol[c_i\kappa_i])\\
& = 4^g \frac{{\disc(\calO_1)}\cdot{\disc(\calO_2)}}{\disc(\ol)^4}.
\end{split}
\end{equation*}
\end{proof}

\begin{cor} \label{crudeGZ} \begin{enumerate}\item Let $A_i$ be an abelian variety with CM by $\calO_{K_i}$. Choose a common field of definition $M$ for $A_1, A_2$ such that $M$ contains the normal closure of both $K_1$ and $K_2$ and both $A_i$ have good reduction over $M$. Let $\gerp$ be a prime ideal of $M$, $(p) = \gerp \cap \ZZ$, and suppose that
\[ A_1 \pmod{\gerp} \cong  A_2 \pmod{\gerp}.\]
Let $r$ be the number of prime ideals $\gerq$ in $\ol$ for which $e(\gerq/p)f(\gerq/p)$ is odd. If $r>0$ then
\[ p \leq \left(4^g  \frac{{\disc_{K_1}}\cdot{\disc_{K_2}}}{\disc(\ol)^4}\right)^{1/r}.\]

\item
Suppose that $[L:\QQ] = 2$, i.e., and that $A_i$ are principally polarized abelian surfaces. Then we have the bound
\[ p \leq \left(16 \frac{{\disc_{K_1}}\cdot{\disc_{K_2}}}{\disc(\ol)^4}\right)^{1/r^\prime},\] according to the following cases (and no other case is possible), where the last columns refer to tables in \cite{GLdenominators}. The first column refers to the decomposition of $p$ in $L$. We use ``s.sing." and ``ssp" to refer to ``supersingular" and ``superspecial", respectively.

\begin{table}[h]
\begin{tabular}{|p{2cm}| p{1.3cm}|p{1.4cm}|p{.4cm}|p{2cm}|p{2.4cm}|p{2.6cm}|}\hline
$p$ & Reduc-tion & Rapoport? &  $r^\prime$ & Table 3\newline ($K$ cyclic) & Table 4\newline ($K$ biquadratic) & Table 5 \newline ($K$ non-Galois)\\
\hline\hline
Unramified (inert/split)& ssp & Yes & 2 &  ii, iv, v & iii, iv, vii, viii & iii, vi, viii, ix, x, xi, xiii, xv, xxii, xxiii
\\ \hline
Inert & s.sing \& not ssp & Yes & 4 &iii & -- & vii
\\ \hline
Ramified & ssp & Yes & 2 &  --& vi & -- 
\\ \hline
Ramified & ssp & No & 1 & vi & ix, x, xi & xvi, xvii, xviii, xix, xx, xxi, xxiv, xxv, xxvi
\\
\hline
\end{tabular}
\caption{The case $[L:\QQ]=2$.}

\end{table}
\end{enumerate}

\end{cor}

\begin{proof} Since the $A_i$ are principally polarized abelian surfaces, they satisfy the Deligne-Pappas condition and, when $p$ is unramified, even the Rapoport condition. We can therefore use the results of \cite{BG, Nicole}.

If $p$ is split in $L$ then every supersingular point is superspecial. In that case, $\End_\ol(A)$ is an order of discriminant $p\ol$ in $B_{p, L}$ and we apply part $(1)$ with $r=2$.

If $p$ is inert, then the reduction is necessarily supersingular, by Lemma~\ref{mustbessing}, and may or may not be superspecial. If it is superspecial, then, again, $\End_\ol(A)$ is an order of discriminant $p\ol$ in $B_{p, L}$ and the bound holds with $r^\prime = 2$. 

If the reduction is supersingular and not superspecial, then in fact $\End_\ol(A)$ has discriminant $p^2\ol$, and so we may take $r^\prime = 4$.

Next we consider the case when $p$ is ramified. There are three case to consider. The first is when we have superspecial reduction and the Rapoport condition holds. In that case, $\End_{\ol}(A)$ has discriminant $p\ol$, and we may take $r^\prime = 2$. The second case is when we have superspecial reduction and the Rapoport condition does not hold (but the Deligne-Pappas condition holds). In this case, $\End_{\ol}(A)$ has discriminant $\gerp$, where $\gerp$ is the prime of $\ol$ above $p$ and we can take $r^\prime=1$. The last possibility is, ostensibly, that we have supersingular reduction, which is not superspecial. This in fact never happens in the presence of CM by the full ring of integers. It is interesting to note, though, that for supersingular and not superspecial reduction,
the abelian variety $A$ has a unique copy of the group scheme $\alpha_p$ contained in it, which is therefore preserved under all endomorphisms. Thus, $\End(A) \injects \End(A/\alpha_p)$ and $A/\alpha_p$ is superspecial, but doesn't satisfy the Rapoport condition (see \cite{AG}).
And so, were this case to occur, we could have taken $r^\prime = 1$. \end{proof}

\begin{rmk} Suppose that $r=0$ then $g$ is even and a maximal order $R\subset B_{p, L}$ has discriminant $1$. For every prime $p$ (and for any decomposition behaviour of $p$), there certainly exist supersingular abelian varieties $A$ with RM such that $\End_{\ol}(A) = R$. This is easily achieved by choosing an $R$-stable lattice of the Dieudonn\'e module of $A$. Experience shows, however, that such abelian varieties tend to be badly behaved, for example, the Deligne-Pappas condition tends to fail when $p$ is unramified, (it fails in the cases we have checked and we did not find an example where it holds) 
or, in other cases, such as when $p$ is totally ramified, the Deligne-Pappas condition holds but the endomorphism ring is not the maximal order.
Thus, one would expect that under the Deligne-Pappas condition the discriminant of $\End_{\ol}(A)$ is never $1$, and, if so, one obtains a version of part (1) of Corollary~\ref{crudeGZ}, in all cases.

In fact, one can be more optimistic and guess that the largest order $\calO$ arising for a supersingular characteristic $p$ abelian variety with RM $A$, satisfying the Deligne-Rapoport condition, also arises for some superspecial such abelian variety. Superspecial abelian varieties with RM were studied by Nicole \cite{Nicole,NicoleJNT}. When $p$ is unramified in $L$ and $A$ is superspecial, $\End_{\ol}(A)$ has discriminant $p\ol$. When $p$ is ramified in $L$, larger orders arise (see \cite[Theorem 2.8.5]{Nicole}), but at least when $p$ is totally ramified, $p\ol = \gerp^{[L:\QQ]}$, still the largest order arising (for a superspecial abelian variety) has discriminant $\gerp$.  
\end{rmk}


\section{Computations: $g=2$} \label{sec:example}

Consider the two primitive Galois quartic CM fields $K=\QQ(\zeta_5)$ and
$K^\prime=\QQ(\sqrt{-85+34\sqrt{5}})$.
The common real quadratic subfield $L = K^+= {K^\prime}^+=\QQ(\sqrt{5})$ has strict class number one, as it has class number one and a unit $(1+\sqrt{5})/2$ of negative norm.
The field $K$ has class number $1$ and the triple of absolute Igusa invariants
of the principally polarized abelian surface with CM by $K$ is $i_1 = i_2 = i_3 =0$.
The field $K^\prime$ has class number $2$ and
the triple of absolute Igusa invariants for one of the CM points associated to $K^\prime$ is:
\begin{center}
$$i_1 = \frac{2^{33} \cdot 3^{10} \cdot 5^5 \cdot 19^5 \cdot 521^5}{71^{12}}, \;
i_2 = \frac{2^{23} \cdot 3^{10} \cdot 5^5 \cdot 19^5 \cdot 521^3}{71^{8}}, \;
i_3 = \frac{2^{16} \cdot 3^{7} \cdot 5^4 \cdot 19^3 \cdot 521^2 \cdot 755777339}{71^{8}}.$$
\end{center}
Genus $2$ curves over $\QQ$ with these invariants are given by the affine models:
$$y^2 =x^5 -1$$ for $\QQ(\zeta_5)$, and
$$ y^2 = -584x^6 - 4020x^5 + 28860x^4 + 130240x^3 - 514920x^2 - 190244x-289455,$$
for $K^\prime$.
In this case, the triple of absolute invariants is insufficient to determine whether 
the two curves are isomorphic modulo a prime $p$, since the first invariant is zero.
To understand for which primes the curves are isomorphic, it is necessary to compute all ten 
Igusa invariants for the CM point associated to $K^\prime$ to determine which primes divide all ten invariants (see~\cite[Section 2.2]{GLdenominators} for an explanation, especially consequence 3 at the end of the subsection).
In particular, primes which divide the differences of all ten Igusa invariants associated to two CM points of $K$ and $K^\prime$ are primes for which the {\it coincidence number} of $K$ and $K^\prime$ defined in Section~\ref{sec: moduli} is non-zero. 

The prime $19$ appears in all three invariants and checking all ten invariants, we find that they too are all zero modulo $19$.
There is also a positive contribution at the prime $p=19$ in our formula in Theorem~\ref{fixedCM}, which implies a non-zero coincidence number.
Since $K$ has class number $1$, there is only one superspecial order $R(\calO,\lambda)$.  We find
an element $x \in \ol$ satisfying condition {\bf C} and count the elements in $S_2(\calO,x)$.
Let $d$ and $d^\prime$ be as in Section~\ref{maincounting}.
We find that for $x= 3\sqrt{5} - 3$, the ideal in $\calO_L$ generated by $(x^2-d d^\prime)/4$ factors as:
\begin{center}
$\gerp_2^2 \gerp_{19,1} \gerp_{19, 2}.$
\end{center}
 We see that there is a positive contribution for $p=19$  in our formula because this factorization has both split factors for $19$, and $2$ is totally inert in $K/L$ but appears to the power $2$, so $(x^2-d d^\prime)/(4 \cdot 19)$ is a norm of an ideal from $K/L$ and the set $S_2(\calO,x)$ is non-empty.

Consider the other primes which are common to all three numerators in this example: $5$ is a ramified prime in $L$, so our results do not cover it; neither do our formulas pertain to the prime $2$ which also appears in all three numerators; the prime $3$ divides all ten invariants but is supersingular, not superspecial, and it certainly satisfies the crude bound Theorem~\ref{crudebound} from Section~\ref{sec:bound}; the prime $521$ does not divide all ten invariants.

\end{document}